\documentclass{amsart}

\usepackage{graphicx}
\usepackage{tikz}
\usetikzlibrary{matrix,arrows,decorations.pathmorphing}
\usepackage{tikz-cd}
\usepackage{hyperref}
\usepackage{amsthm}
\usepackage[mathscr]{euscript}
\usepackage{amssymb}
\usepackage{mathrsfs}
\usepackage{enumitem}
\usepackage[pagewise]{lineno}
\newtheorem{theorem}{Theorem}[section]
\newtheorem*{theorem*}{Theorem}
\newtheorem{lemma}[theorem]{Lemma}
\newtheorem*{conjecture*}{Conjecture}
\newtheorem{theoremM}{Theorem}

\theoremstyle{definition}
\newtheorem{definition}[theorem]{Definition}

\newtheorem{proposition}[theorem]{Proposition}

\newtheorem{corollary}[theorem]{Corollary}

\theoremstyle{remark}
\newtheorem{remark}[theorem]{Remark}

\numberwithin{equation}{section}

\usepackage{color}

\usepackage[all,cmtip]{xy}

\newcommand{\F}{\mathbb{F}}

\newcommand{\Z}{\mathbb{Z}}

\newcommand{\C}{\mathbb{C}}
\newcommand{\Gm}{\mathbb{G}_m}

\newcommand{\opn}{\operatorname}

\newcommand{\BP}{\opn{BP}}

\newcommand{\oP}{\opn{P}}

\newcommand{\cl}{\opn{cl}}
\newcommand{\CH}{\opn{CH}}
\newcommand{\He}{H_{\acute{e}t}}
\newcommand{\Sp}{\opn{Sp}}
\newcommand{\Mot}{\mathbf{HMot}^{\mathbb{C}}_{\bullet}}
\newcommand{\Motk}{\mathbf{HMot}^k_{\bullet}}

\newcommand{\SMotk}{\mathbf{SHM}^k}

\newcommand{\Top}{\mathbf{HTop}_{\bullet}}
\newcommand{\Sg}{\mathbf{Sing}}

\newcommand{\STop}{\mathbf{SHT}}
\newcommand{\Sm}{\mathbf{Sm}^{\mathbb{C}}}
\newcommand{\Smk}{\mathbf{Sm}^k}

\newcommand{\specC}{\opn{Spec}\mathbb{C}}

\newcommand{\brxi}{\bar{\xi}}
\newcommand{\breta}{\bar{\eta}}
\newcommand{\re}{\opn{res}}
\newcommand{\tr}{\opn{tr}}
\newcommand{\fR}{\mathfrak{R}}
\newcommand{\RomanNumeralCaps}[1]
{\MakeUppercase{\romannumeral #1}}
\newcommand{\Gu}{\textrm{the author}}

\begin{document}
	\title[The Chow ring and Cohomology of $BPGL_n$]{A distinguished subring of the Chow ring and cohomology of $BPGL_n$}
	\author{Xing Gu}
	\address{Max Planck Institute for Mathematics, Vivatsgasse 7, 53111 Bonn, Germany,\\
			Hebei Normal University, Shijiazhuang, 050024}
	\email{gux2006@mpim-bonn.mpg.de}

	\thanks{The author thanks Max Planck Institute for Mathematics, Bonn for their hospitality and supports in various ways.}
	
	
	
	\date{}
	
	\dedicatory{}
	
	
	\begin{abstract}
		We determine a subring of the Chow ring and the cohomology of $BPGL_n$, the classifying space of the projective linear group of degree $n$ over complex numbers, and explain a way in which this computation might play a role in the period-index problem. In addition, we show that the Chow ring of $BPGL_n$ is not generated by the Chern classes of complex linear representations of $PGL_n$.
	\end{abstract}
	
	\maketitle
	\tableofcontents
	\section{Introduction}\label{sec:intro}
	The cohomology of classifying spaces of Lie groups is among the fundamental subjects in topology. A similar role in algebraic geometry is played by the Chow rings of the classifying spaces of algebraic groups over a field, defined by Totaro \cite{totaro1999chow}, which may alternatively be described in terms of motivic cohomology. In this paper we consider the Chow ring and cohomology of the classifying space of the complex projective linear group.
	
	\subsection*{Notations} Throughout this paper, we adopt the following notations:
	\begin{itemize}
		\item $H_M^{s,t}(X;R)$: the motivic cohomology group of bidegree $(s,t)$ for a motivic space $X$ with coefficients in a commutative unital ring $R$, where the term ``motivic space'' is defined in Section \ref{sec:mot};
		\item $\He^s(X;\mathcal{F})$: the {\'e}tale cohomology of an {\'e}tale sheaf $\mathcal{F}$ over a scheme $X$.
		\item $H^s(Y;R)$: the singular cohomology group of degree $s$ for a topological space $Y$ with coefficients in $R$;
		\item $H_M^{s,t}(X)=H_M^{s,t}(X;\Z)$,\ $H^s(Y)=H^s(Y;\Z)$;
		\item $BG$: the classifying space of a Lie group $G$, or the geometric classifying space of an algebraic group $G$ which is discussed Section \ref{sec:mot};
		\item $\CH^t(X):=H_M^{2t,t}(X)$: the Chow group of degree $t$ for $X$ a smooth scheme over $\C$ or $X=BG$ for $G$ an algebraic group, or equivalently Totaro's Chow ring of $BG$ defined in \cite{totaro1999chow}.
		\item $\cl: H_M^{s,t}(X)\to H^{s}(X(\C))$: the (complex) cycle class map for $X$ a smooth scheme over $\C$ , and $X(\C)$ the manifold of complex points of $X$, or, in the sense of Totaro \cite{totaro1999chow}, for $X=BG$ where $G$ is an algebraic group over $\C$ and $X(\C)=BG(\C)$ for $G(\C)$ the Lie group of complex points of $G$. This is discussed in Section \ref{sec:mot}. In the case of Chow rings, we have $\cl: \CH^t(X)\to H^{2t}(X(\C))$ which is the cycle class map in the classical sense.
		\item $GL_n:=GL_n(\C)$ and $SL_n:=SL_n(\C)$: the general liniear group and the special linear group of degree $n$ over $\C$;
		\item $PGL_n:=GL_n/\C^{\times}$: the projective lienar group of degree $n$ over $\C$, i.e., $GL_n$ modulo its center, the subgroup of invertible scalar matrices;
		\item $PU_n:=U_n/S^1$: the projective unitary group of order $n$, i.e., the unitary group $U_n$ modulo its center.
		\item $K(R,s)$: the Eilenberg-Mac Lane space representing the cohomology functor $H^s(-;R)$ for a commutative unital ring $R$.
		
	\end{itemize}
	In the case of singular cohomology, we always consider $BU_n$ and $BPU_n$ instead of $BGL_n$ and $BPGL_n$, since $U_n$ and $PU_n$ are respectively the maximal compact subgroups of $GL_n$ and $PGL_n$.
	
	Among the Chow rings $\CH^*(BG)$ and $H^*(BG)$, the case $G=PGL_n$ (or $G=PU_n$) is one of the most difficult, as pointed out by Molina Rojas and Vistoli \cite{molina2006chow}, in which a unified approach is provided to the Chow rings of classifying spaces for many classical groups, not including $PGL_n$.
	
	On the other hand, the case for $PGL_n$ is potentially of the richest structure. For instance, the torsion classes in $\CH^*(BPGL_n)$ and $H^*(BPU_n)$ are all $n$-torsions, by Proposition 2.3 of \cite{vezzosi2000chow}.
	
	In addition to the significance of $BPGL_n$ and $BPU_n$ in their own rights, the cohomology of $BPU_n$ has applications in the topological period-index problem \cite{antieau2014topological}, \cite{gu2019topological} and the study of anomalies in physics \cite{cordova2020anomalies}, \cite{garcia2019dai}.
	
	The cohomology algebra $H^*(BPU_{4n+2};\F_2)$ is determined by Kono and Mimura \cite{kono1975cohomology} and Toda \cite{toda1987cohomology}. The cohomology algebra $H^*(BPU_3;\F_3)$ is determined by Kono, Mimura, and Shimada \cite{kono1975cohomology1}. Vavpeti{\v{c}} and Viruel \cite{vavpetivc2005mod} show some properties of $H^*(BPU_p;\F_p)$ for an arbitrary odd prime $p$.
	
	The Chow ring $\CH^*(BPGL_3)$ is almost determined by Vezzosi \cite{vezzosi2000chow}, which is subsequently improved by Vistoli \cite{vistoli2007cohomology}, which completes the study of $\CH^*(BPGL_3)$ and determined the additive structure as well as a large part of the ring structures of $\CH^*(BPGL_p)$ and $H^*(BPU_p)$, for $p$ an odd prime. The Brown-Peterson cohomology of $BPU_p$ for an odd prime $p$ is determined by Kono and Yagita \cite{kono1993brown}.
	
	The author \cite{gu2019cohomology} determines the ring structure of $H^*(BPU_n)$ for any $n>0$ in dimensions less than or equal to $10$, and obtains partial results on the Chow ring and the Brown-Peterson cohomology of $BPGL_n$ in \cite{gu2019some} and \cite{gu2020brown}.
	
	In \cite{gu2019cohomology}, the author considers a map
	\begin{equation}\label{eq:intro-chi}
		\chi:BPU_n\to K(\Z,3),
	\end{equation}
	and the image of the induced homomorphism
	\[\chi^*:H^*(K(\Z,3))\to H^*(BPU_n).\]
	in which we have classes
	\[y_{p,k}\in H^{2p^{k+1}+2}(BPU_n),\ k\geq 0\]
	which are nontrivial $p$-torsion classes for $p\mid n$ and trivial otherwise. In the case $p\mid n$ and $p^2\nmid n$, the author \cite{gu2019some} shows that there are  $p$-torsion classes
	\[\rho_{p,k}\in \CH^{p^{k+1}+1}(BPGL_n),\ k\geq 0\]
	satisfying $\cl(\rho_{p,k})=y_{p,k}$.
	
	However, the author \cite{gu2019some} does not show anything about $\CH^*(BPGL_n)$ for $n$ with $p$-adic valuation greater than $1$. Here, the \emph{$p$-adic valuation} of $n$ means the greatest integer $r$ satisfying $p^r\mid n$.
	
	The classes $\rho_{p,k}$ and $y_{p,k}$ are ``periodic'' in the following sense. Suppose we have $p\mid m\mid n$. Then we have the obvious diagonal homomorphism
	\begin{equation}\label{eq:diag_hom}
		\Delta: PGL_m \to PGL_n,\
		[A] \to
		\begin{bmatrix}
			A & &\\
			&\ddots &\\
			& & A
		\end{bmatrix},
	\end{equation}
	which induces a map of classifying spaces $B\Delta: BPGL_m\to BPGL_n$. Then it is shown in \cite{gu2019some} that we have $B\Delta^*(y_{p,k})=y_{p,k}$. If in addition we have $p^2\nmid n$, then $B\Delta^*(\rho_{p,k})=\rho_{p,k}$.
	
	Despite the works discussed above, very little has been understood about the role of the $p$-adic valuation of $n$ in  $\CH^*(BPGL_n)$ and $H^*(BPU_n)$. The purpose of this paper is to offer some insight into this, in the form of the following two theorems.
	\begin{theoremM}\label{thm:main}
		Let $p$ be an odd prime, and $n$ a positive integer divisible by $p$. Then there are nontrivial $p$-torsion classes \[\rho_{p,k}\in\CH^{p^{k+1}+1}(BPGL_n),\ \ y_{p,k}=\cl(\rho_{p,k})\in H^{2p^{k+1}+2}(BPU_n)\]
		for $k\geq 0$, such that for $p\mid m\mid n$ and $\Delta: PGL_m\to PGL_n$, we have
		\begin{equation}\label{eq:periodic}
			\begin{cases}
				B\Delta^*(\rho_{p,k})=\rho_{p,k},\\
				B\Delta^*(y_{p,k})=y_{p,k}.
			\end{cases}
		\end{equation}
		
		Furthermore, suppose $r\geq1$ is the $p$-adic valuation of $n$. Then there are injective ring homomorphisms
		\begin{equation}\label{eq:rho-p-k}
			\Z[Y_k\mid 0\leq k\leq 2r-1]/(pY_k)\hookrightarrow\CH^*(BPGL_n),\ Y_k\mapsto\rho_{p,k},
		\end{equation}
		and
		\begin{equation}\label{eq:y-p-k}
			\Z[Y_k\mid 0\leq k\leq 2r-1]/(pY_k)\hookrightarrow H^*(BPU_n),\ Y_k\mapsto y_{p,k}.
		\end{equation}
	\end{theoremM}
	Notice that, away from degree $0$, the ring
	\[\Z[Y_k\mid 0\leq k\leq 2r-1]/(pY_k)\]
	is isomorphic to a graded the polynomial ring $\F_p[Y_k\mid 0\leq k\leq 2r-1]$, with the degree of $Y_{p,k}$ equal to $p^{k+1}+1$ in the case of Chow rings, or $2p^{k+1}+2$, in the case of singular cohomology.
	
	For each $n>1$, we define the subrings
	\begin{equation*}
		\begin{cases}
			\fR_M(n) = \Z[\rho_{p,k} \mid k\geq 0]/(p\rho_{p,k}) \subset \CH^*(BPGL_n),\\
			\fR(n) = \Z[y_{p,k} \mid k\geq 0]/(py_{p,k}) \subset H^*(BPGL_n).
		\end{cases}
	\end{equation*}
	
	\begin{theoremM}\label{thm:subringR}
		Let $p$ be an odd prime and $n>1$ an integer with $p$-adic valuation $r>0$. Then the homomorphisms $B\Delta^*$ restrict to isomorphisms
		\begin{equation*}
			\begin{cases}
				B\Delta^*: \fR_M(n)\xrightarrow{\cong} \fR_M(p^r) ,\ \rho_{p,k}\mapsto \rho_{p,k},\\
				B\Delta^*: \fR(n)\xrightarrow{\cong} \fR(p^r) ,\ y_{p,k}\mapsto y_{p,k}.
			\end{cases}
		\end{equation*}
	\end{theoremM}
	
	In Theorem \ref{thm:main}, the condition $0\leq k\leq 2r-1$ in \eqref{eq:rho-p-k} and \eqref{eq:y-p-k} is essential at least when $n$ is of $p$-adic valuation $1$, as shown in the following
	\begin{theoremM}\label{thm:example}
		For $p$ and odd prime, and $n>0$ an integer satisfying $p\mid n$ and $p^2\nmid n$, there are nontrivial polynomial relations in the rings $\fR_M(n)$ and $\fR(n)$ as follows:
		\begin{equation}\label{eq:polynomial-relation-rho}
			\rho_{p,0}^{p^2+1}+\rho_{p,1}^{p+1}+\rho_{p,0}^p\rho_{p,2}=0,
		\end{equation}
		\begin{equation}\label{eq:polynomial-relation-y}
			y_{p,0}^{p^2+1}+y_{p,1}^{p+1}+y_{p,0}^py_{p,2}=0.
		\end{equation}
	\end{theoremM}
	
	\begin{remark}
		There are inclusions of $p$-elementary abelian subgroups
		\[\theta: V^{2r}\hookrightarrow PGL_{p^r}\]
		such that when $r=1$, the homomorphism $B\theta^*$ is injective when restricted to the torsion subgroup of $\CH^*(BPGL_p)$ (Lemma \ref{lem:BPGLp-inj}). The polynomial relations \eqref{lem:BPGLp-inj} and \eqref{lem:BPGLp-inj} are therefore detected by the relations in $\CH^*(BV^2)$. For a general $r$, we hope that similar polynomial relations may be detected by $\CH^*(BV^{2r})$.
	\end{remark}
	
	\subsection{Outline of proofs} The classes $y_{p,k}$ are constructed in \cite{gu2019some}, which we recall in this paper. To construct the classes $\rho_{p,k}$, we define a class $\zeta_1\in H_M^{3,2}(BPGL_n)$ via {\'e}tale cohomology and the Beilinson-Lichtenbaum conjecture. The classes $\rho_{p,k}$ are constructed by applying Steenrod reduced power operations to the class $\zeta_1$.
	
	To verify the injectivity of the homomorphisms \eqref{eq:rho-p-k} and \eqref{eq:y-p-k}, it suffices to verify the latter, from which the former follows via the cycle class map. We reduce it to the case $n=p^r$ and consider an inclusion of a non-toral elementary abelian $p$-subgroup
	\[\theta:V^{2r}\to PU_{p^r},\]
	and show that the composition
	\begin{equation*}
		\begin{split}
			\Z[Y_k\mid 0\leq k\leq 2r-1]/(pY_k)&\rightarrow H^*(BPU_{p^r})\xrightarrow{B\theta^*} H^*(BV^{2r})\\
			Y_k&\mapsto y_{p,k}
		\end{split}
	\end{equation*}
	is injective.
	
	Theorem \ref{thm:example} follows from Vistoli \cite{vistoli2007cohomology} and some additional computation involving the transfer maps
	\begin{equation*}
		\begin{cases}
			\tr^H_G:\CH^*(BH)\to\CH^*(BG),\\
			\tr^H_G:H^*(BH)\to H^*(BG)
		\end{cases}
	\end{equation*}
	for $H$ a subgroup of $G$ of finite index.
	
	\subsection{The period-index problem} The classical version of the period-index problem (\cite{gille2017central}, \cite{grothendieck1968groupe}) concerns a field $k$ and the degrees of central simple algebras over $k$ and its Brauer group, or more generally the Brauer group to a scheme and the degrees of Azumaya algebras over it. In \cite{antieau2014period}, Antieau and Williams initiated the study of a topological analog of the period-index problem, which we call the topological period-index problem.
	
	The cohomology of $BPU_n$ plays an important role in the study of the topological period-index problem. In this paper we briefly discuss how $\CH^*(BPGL_n)$ may play a similar role in the period-index problem for schemes.
	
	\subsection{The Chern subrings} We have an interesting consequence of Theorem \ref{thm:main}, regarding the \textit{Chern subrings}.
	\begin{definition}\label{def:Chern-subring}
		For $G$ an algebraic group over $\C$, and a commutative unital ring $R$, the Chern subring of $\CH^*(BG))\otimes R$ is the subring generated by Chern classes of all representations of $\varphi:G\rightarrow GL_r$ for some $r$, i.e., the image of the pull-back homomorphisms
		\[B\varphi^*:\CH^*(BGL_r)\otimes R\cong R[c_1,\cdots,c_r]\to\CH^*(BG)\otimes R.\]
		If the Chern subring is equal to $\CH^*(BG)\otimes R$, then we say that $\CH^*(BG)\otimes R$ is \textit{generated by Chern classes}. The Chern subrings for any generalized cohomology theories of $BG$ are similarly defined.
	\end{definition}
	For an abelian group $A$, let $A_{(p)}$ denote the localization of $A$ at $p$, or equivalently, tensor product with $\Z_{(p)}$. Vezzosi \cite{vezzosi2000chow} shows that $\CH^*(BPGL_3)_{(3)}$ is not generated by Chern classes.  The same is shown for $\CH^*(BPGL_p)_{(p)}$ for all odd primes $p$ independently by Kameko and Yagita \cite{kameko2010chern}, and Targa \cite{targa2007chern}, and is shown for $\CH^*(BPGL_n)_{(p)}$ with $p\mid n$ and $p^2\nmid n$ by the author \cite{gu2019some}. The same result for the Brown-Peterson cohomology $\BP^*(BPGL_p)$ is proved in Kono and Yagita \cite{kono1993brown}. It is shown in \cite{gu2019some} and \cite{gu2020brown}, respectively, that $H^*(BPGL_n)_{(p)}$ and $\BP^*(BPGL_n)$ are not generated by Chern classes for $p\mid n$. We extend the above mentioned results for $\CH^*(BPGL_n)_{(p)}$ to the most general case:
	\begin{theoremM}\label{thm:Chern-subring}
		Let $n>1$ be an integer, and $p$ one of its odd prime divisor. Then the Chow ring $\CH^*(BPGL_n)_{(p)}$ is not generated by Chern classes. More precisely, the class $\rho_{p,0}^i$ is not in the Chern subring for $p-1\nmid i$.
	\end{theoremM}
	
	\subsection{Organization of the paper} Section \ref{sec:mot} is a brief review motivic homotopy theory required in the rest of this paper. In Section \ref{sec:rho} we recall the definition of the classes $y_{p,k}$ in \cite{gu2019some}, and construct the classes $\rho_{p,k}$. In Section \ref{sec:extraspectial} we prove a lemma on the cohomology of an extraspecial $p$-group, which plays a key role in the construction of the non-toral $p$-elementary subgroup $V^{2r}$ of $PU_{p^r}$. Then we study the cohomology of $BV^{2r}$ in Section \ref{sec:p-elementary}, where we complete the proof of Theorem \ref{thm:main}. In Section \ref{sec:polynomial} we prove Theorem \ref{thm:example}. Section \ref{sec:tpip} is a brief discussion on the period-index problem. In Section \ref{sec:Chern} we discuss the Chern subrings and prove Theorem \ref{thm:Chern-subring}. In the appendix we discuss a Jocobian criterion for algebraic independence over perfect fields, which is used in Section \ref{sec:p-elementary}.
	
	\subsection{Acknowledgement} The author is grateful to Burt Totaro for pointing out an error in an earlier version, and for other helpful conversations.
	
	The author is grateful to Zhiyou Wu for communications on basic knowledge on motivic homotopy theory, to Benjamin Antieau, Diarmuid Crowley, Christian Haesemeyer, and Ben Williams for their continuing interests in the study of $BPGL_n$, to Kasper Anderson, Denis Nardin, Oliver R{\"o}ndigs, David E Speyer, and Mathias Wendt for discussions via MathOverflow, to the anonymous referees who offered various pieces of advice that improve the paper. Finally the author would like to thank the Max Planck Institute for Mathematics in Bonn for various supports amid the COVID-19 global pandemic.
	\section{Preliminaries on motivic homotopy theory}\label{sec:mot}
	In this section we review the necessary backgrounds in motivic cohomology and homotopy theory.  Let $\Smk$ be the category of smooth schemes over a field $k$,  and
	\[\mathbf{Mot}^k_{\bullet}:=\Delta^{op}\opn{PShv}_{\bullet}(\Smk)\]
	be the category of simplicial presheaves over $\Smk$.
	\begin{remark}\label{rem:simplicial-sheaves}
		In general, we let $\Delta^{op}\opn{PShv}_{\bullet}(\mathscr{C})$ denote the category of pointed simplicial sheaves over a category $\mathscr{C}$, and let $\Delta^{op}\opn{Shv}_{\bullet}(\mathscr{S})$ denote the category of pointed simplicial sheaves over a site $\mathscr{S}$.
	\end{remark}
	Moreover, let $\mathbf{Top}$ ($\mathbf{Top}_{\bullet}$) be the category of (pointed) locally contractible topological spaces. The categories $\mathbf{Mot}^k_{\bullet}$ and $\mathbf{Top}_{\bullet}$ are enriched over themselves, and we denote the mapping spaces by $\opn{Map}_{\mathbf{Mot}^k_{\bullet}}(-,-)$ and $\opn{Map}_{\mathbf{Top}_{\bullet}}(-,-)$. We call objects of $\mathbf{Mot}^k_{\bullet}$ motivic spaces.
	
	We consider the pointed motivic homotopy category $\Motk$ over the base field $k$, which is the homotopy category of the category of simplicial presheaves  $\Delta^{op}\opn{PShv}_{\bullet}(\Smk)$, taking Bousfield localization with respect to the Nisnevich hypercovers and the canonical projections $X\times \mathbf{A}^1\to X$, where $\mathbf{A}^1$ is the affine line. We also consider the homotopy category of pointed locally contractible topological spaces $\Top$. 
	\begin{remark}
		We choose to take $\Delta^{op}\opn{PShv}_{\bullet}(\Smk)$ as the ambient category of motivic spaces, instead of $\Delta^{op}\opn{Shv}_{\bullet}(\Smk_{Nis})$, where $\Smk_{Nis}$ is the Nisnevish site over $\Smk$, as done by Morel and Voevodsky \cite{morel19991}. The resulting homotopy categories are the same, as explained in \cite{dugger2001universal}, for instance. Our choice of simplicial presheaves makes it slightly easier for arguments on monoidal structures.
	\end{remark}
	\subsection{The motivic stable homotopy category}\label{ssec:Mot-Sus-Loop}
	In the category $\Top$, we have the suspension functors $\Sigma^s=S^s\wedge -$, where $S^s$ is the $s$-dimensional sphere, and $\wedge$ is the smash product. 
	
	In the category $\Motk$, we have smash products defined by the object-wise smash product of simplicial presheaves. The notion of spheres in $\Motk$ is slightly complicated. We define the simplicial circle $S^{1,0}:=\Delta^1/\partial\Delta^1$, where the simplicial sets are regarded as constant simplicial presheaves, and the Tate circle $S^{1,1}:=\Gm$, where $\Gm$ is the algebraic group $\opn{Spec}k[x^{\pm 1}]$. We therefore have spheres
	\begin{equation*}
		S^{s,t}:=(S^{1,0})^{\wedge s-t}\wedge (S^{1,1})^{\wedge t}
	\end{equation*}
	for $s\geq t$, and the bigraded suspension functors
	\begin{equation}\label{eq:Mot-suspensions}
		\Sigma^{s,t}:=S^{s,t} \wedge -:\Motk\to\Motk.
	\end{equation}
	
	By ``formally inverting the suspension fucntors'', we obtain the stabilization of $\Top$ and $\Motk$, which we denote by $\STop$ and $\SMotk$, respectively. We call objects of $\STop$ spectra, and objects of $\SMotk$ motivic spectra. For the construction of $\SMotk$, see \cite{dundas2007motivic}. In both the topological and motivic cases, we have the stabilization functors
	\begin{equation*}
		\begin{cases}
			\Sigma^{\infty}: \Top\to\STop,\\
			\Sigma^{\infty}_M:\Motk\to\SMotk.
		\end{cases}
	\end{equation*}
	
	\subsection*{The motivic Eilenberg-Mac Lane spaces and spectra}\label{ssec:Mot-EM}
	For a commutative unital ring $R$, consider the Eilenberg-Mac Lane motivic spectrum $H_MR$ representing $H_M^{*,*}(-;R)$, i.e., for a smooth scheme $X$ over $k$, we have natural isomorphisms of groups
	\begin{equation}\label{eq:stable-representable}
		\opn{Hom}_{\SMotk}( \Sigma^{\infty}_MX, \Sigma^{*,*}H_MR)\cong H_M^{*,*}(X;R).
	\end{equation}
	The left-hand side is canonically an abelian group, as $\SMotk$ is a triangulated category. The notation $H_MR$ is set to be distinguished from $HR$, the classical Eilenberg-Mac Lane spectrum in $\STop$. For $s\geq t\geq 0$, we have motivic Eilenberg-Mac Lane spaces $K(R(t),s)$ which are abelian group objects of $\Motk$, 
	representing the motivic cohomology functor $H_M^{s,t}(-;R)$, i.e., for a smooth scheme $X$ over $\C$, we have natural isomorphisms
	\begin{equation}\label{eq:representable}
		\opn{Hom}_{\Motk}(X,K(R(t),s))\cong H_M^{s,t}(X;R).
	\end{equation}
	See \cite{hoyois2015algebraic} for the construction of $K(R(t),s)$. We may extend the definition of motivic cohomology to a functor from $\Motk$, by letting $X$ at the left-hand-side of \eqref{eq:representable} be any object in the category $\Motk$.

	The ring structure of motivic cohomology yields a morphism
	\begin{equation}\label{eq:Motivic-EZ}
		\mathfrak{m}_M:K(R(t),s)\wedge K(R(l),k)\to K(R(t+l),s+k).
	\end{equation}	
	Passing to the stable homotopy category $H_MR$, we have the following
	\begin{proposition}\label{pro:ring-spectra}
		For $R$ a commutative unital ring, $H_MR$ is a motivic commutative ring spectra, i.e., we have a unital, commutative, associative morphism
		\[\mathfrak{m}_M:H_MR\wedge H_MR\to H_MR\]
		which gives the product of motivic cohomology.
	\end{proposition}
	For the short exact sequence of $\Z$-modules $\Z\xrightarrow{\times n}\Z\rightarrow\Z/n$, the associated long exact sequence of motivic cohomology groups yields a Puppe sequence
	\[\cdots\to K(\Z(t),s)\xrightarrow{\times n} K(\Z(t),s)\to K(\Z/n(t),s)\to K(\Z(t),s+1)\to\cdots\]
	in which every two consecutive arrows yield a fiber sequence of spaces. The last arrow represents the Bockstein homomorphism
	\begin{equation}\label{eq:Mot-Bockstein}
		\delta: H_M^{s,t}(-;\Z/n)\to H_M^{s+1,t}(-;\Z).
	\end{equation}

	Passing to the stable homotopy category, the Bockstein homomorphism above yields a morphism
	\begin{equation}\label{eq:Mot-Bockstein-stable}
		\delta: H_M(\Z/n)\to \Sigma^1 H_M\Z.
	\end{equation}
	
	\subsection{The $\C$-realization functor}\label{ssec:realization}
	Consider the functor
	\begin{equation}\label{eq:realizing-schemes}
		\Sm\to\mathbf{Top}_{\bullet},\ X\mapsto X(\C)_+
	\end{equation}
	of taking complex points with a disjoint base point. Let $\Delta^{op}\mathbf{Sets}_{\bullet}$ be the category of pointed simplicial sets. 
	
	For a pointed topological space $Y$, let $\Sg(Y)$ be the pointed simplicial set of singular complexes of $Y$, i.e., we have
	\begin{equation*}
		\Sg(Y)_n=\opn{Hom}_{\mathbf{Top}_{\bullet}}(\Delta^n,Y)
	\end{equation*}
	with the obvious face and degeneracy maps, and $\Delta^n$ the standard topological simplices. Then we have a functor
	\begin{equation}\label{eq:singular-realizing-schemes}
		\Sm\to\Delta^{op}\mathbf{Sets}_{\bullet},\ X\mapsto\Sg(X(\C)).
	\end{equation}
	We take the left homotopy Kan extension of \eqref{eq:singular-realizing-schemes} and obtain a functor
	\begin{equation}\label{eq:pre-tC}
		\mathbf{Mot}^{\C}_{\bullet}=\Delta^{op}\opn{PShv}_{\bullet}(\Sm)\to\Delta^{op}\mathbf{Sets}_{\bullet},
	\end{equation}
	which is a left Quillen functor. We denote the total left derived functor by
	\[t^{\C}:\Mot\to\Top\]
	which we also call the $\C$-realization functor, noticing that the homotopy category of $\Delta^{op}\mathbf{Sets}_{\bullet}$, with the classical model structure, is well known to be equivalent to $\Top$ (\cite{quillen2006homotopical}). We make the choice of $\Delta^{op}\mathbf{Sets}_{\bullet}$ over $\mathbf{Top}_{\bullet}$ as the target category since the former is easier for comparison with simplicial $R$-modules.
	\begin{remark}\label{rmk:left-Kan-model}
		We may take, for instance, the following model for the left homotopy Kan extension:
		\begin{equation*}
			\Delta^{op}\opn{PShv}_{\bullet}(\Sm)\to\Delta^{op}\mathbf{Sets}_{\bullet},\
			\mathcal{F}\to \{[n]\mapsto\mathcal{F}(\Delta^n_{\C})_n \},
		\end{equation*}
		where
		\begin{equation*}
			\{\Delta^n_{\C} := \specC[t_0,\cdots,t_n]/(\sum_i t_i-1)\}_{n\geq 0}
		\end{equation*}
		is a collection of cosimplicial objects with the obvious co-face and co-degeneracy maps.
	\end{remark}
	As explained in Section 3.3 of \cite{morel19991}, the functor $t^{\C}$ takes a presheaf represented by a simplicial smooth scheme $\mathcal{X}$ to the geometric realization of $\mathcal{X}(\C)$, the simplicial topological space of degree-wise complex points of $\mathcal{X}$. Therefore, we have
	\begin{equation}\label{eq:tC-spheres}
		t^{\C}(S^{s,t})=S^s.
	\end{equation}
	
	It is shown in \cite{panin2009voevodsky}, Theorem A.23, that $t^{\C}$ is a strict symmetric monoidal Quillen functor, where the strict symmetric monoidal structure on $\mathbf{Top}_{\bullet}$ is given by smash products. It is shown in \cite{hoyois2015algebraic}, Theorem 5.5, that $t^{\C}(H_MR)\cong HR$. Therefore, we have
	\begin{lemma}\label{lem:realization-ring}
		Let $R$ be a commutative unital ring. For the (motivic) Eilenberg-Mac Lane spectra, we have $t^{\C}(H_MR)\cong HR$ and $t^{\C}(\mathfrak{m}_M)=\mathfrak{m}$.
	\end{lemma}
	As an immediate consequence of Lemma \ref{lem:realization-ring} and \eqref{eq:Mot-Bockstein-stable}, we have the following
	\begin{proposition}\label{pro:product-realization}
		For $R$ is a commutative, unital ring, the natural transformation
		\[\cl:H_M^{s,t}(-;R)\to H^s(t^{\C}(-);R)\]
		is compatible with the ring structures of both sides, and compatible with the Bockstein homomorphisms of both motivic and singular cohomology. 
		
	\end{proposition}

	
	\subsection{The Beilinson-Lichtenbaum Conjecture}\label{ssec:etale}
	It is shown in \cite{artin1972seminaire} that when the base field $k$ has characteristic prime to $n$, any locally constant torsion {\'e}tale abelian sheaf with torsion order $n$ is invariant under base changes along $\mathbf{A}^1$. As an immediate consequence, we have
	\begin{proposition}[\cite{mazza2011lecture}, Corollary 9.25]\label{pro:mu_ninv}
		For a base field $k$ of characteristic prime to $n$. Then any locally constant torsion {\'e}tale abelian sheaf with torsion order $n$ is $\mathbf{A}^1$-local. 
	\end{proposition}
	With the Lichtenbaum cohomology $H_L^{s,t}(-;\Z/n)$(\cite{voevodsky2003motivic}, Definition 10.1) acting as a bridge, this enables the construction of the ``{\'e}tale cycle class map'', i.e.,  the natural transformation
	\begin{equation}\label{eq:et-cl}
		\cl_{\acute{e}t}: H_M^{s,t}(-;\Z/n)\to H_L^{s,t}(-;\Z/n)\cong\He^s(-;\mu_n^{\otimes t})
	\end{equation}
	when the characteristic of the base field $k$ is prime to $n$.

	The following theorem is known as the Beilinson-Lichtenbaum Conjecture:
	\begin{theorem}[Voevodsky, Theorem 6.17, \cite{mazza2011lecture}]\label{thm:BL}
		For smooth schemes over a field $k$ and $n$ be an integer prime to the characteristic of $k$, and nonnegative integers $s\leq t$, the homomorphism \eqref{eq:et-cl} is an isomorphism.
	\end{theorem}
	When $k=\C$, we have the inclusions 
	\[\textrm{Nisnevich covers}\subset\textrm{{\'e}tale covers}\subset\textrm{local homoemorphisms},\]
	which yields the following
	
	\begin{proposition}\label{pro:factor-cl}
		Let $R$ be a commutative unital ring and $X$ be a complex smooth scheme. The complex cycle class map $\cl$ factors, functorial in $X$, as
		\begin{equation*}
			\cl: H_M^{s,t}(X;R)\xrightarrow{\cl_{\acute{e}t}}H_L^{s,t}(X;R)\rightarrow H^s(X(\C);R)
		\end{equation*}
		where $X(\C)$ is the underlying complex manifold of $X$. 
		For $R=\Z/n$, the second arrow is identified, via the identification $\Z/n\cong(\Z/n)^{\otimes t}$ and Theorem \ref{thm:BL}, to the usual comparison map
		\begin{equation*}
			H_{\acute{e}t}^s(-;\mu_n^{\otimes t})\to H^s(-;\Z/n).
		\end{equation*}
	\end{proposition}
	\subsection{The motivic Steenrod reduced power operations}\label{ssec:Mot-Steenrod}
	In \cite{voevodsky2003reduced}, Voevodsky constructs stable operations satisfying a set of axioms and Adem relations similar to those of the Steenrod reduced power operations for singular cohomology.
	
	Let $p$ be an odd prime and $\F_p$ be the field of order $p$. Then the motivic Steenrod reduced power operations are:
	\begin{equation*}
		\begin{split}
			&\beta: H_M^{s,t}(-;\F_p)\to H_M^{s+1,t}(-;\F_p),\\
			&\oP^i: H_M^{s,t}(-;\F_p)\to H_M^{s+2i(p-1),t+i(p-1)}(-;\F_p).
		\end{split}
	\end{equation*}
	The reader may refer to \cite{voevodsky2003reduced} for the Adem relations.
	\begin{remark}
		As in the case of classical Steenrod operations, the operation $\beta$ is the composition of the Bockstein homomorphism $\delta$ and the mod $p$ reduction:
		\[\beta: H_M^{s,t}(-;\F_p)\xrightarrow{\delta}H_M^{s+1,t}(-;\Z)\xrightarrow{\mod p} H_M^{s+1,t}(-;\F_p).\]
	\end{remark}
	\begin{remark}
		The notations above coincide with those of the classical Steenrod operations, which will appear in this paper as well. It will be made clear by the context which is intended.
	\end{remark}
	
	The motivic Steenrod operations are compatible with the classical ones in the following sense. As pointed out in 3.11 of \cite{voevodsky1999voevodsky}, for $k=\C$, we have the commutative diagrams
	\begin{equation}\label{eq:Steenrod-beta-diag}
		\begin{tikzcd}
			H_M^{s,t}(X;\F_p)\arrow[d,"\cl"]\arrow[r,"\beta"]&H_M^{s+1,t}(X;\F_p)\arrow[d,"\cl"]\\
			H^s(t^{\C}(X);\F_p)\arrow[r,"\beta"]&H^{s+1}(t^{\C}(X);\F_p),
		\end{tikzcd}
	\end{equation}
	and
	\begin{equation}\label{eq:Steenrod-P-diag}
		\begin{tikzcd}
			H_M^{s,t}(X;\F_p)\arrow[d,"\cl"]\arrow[r,"\oP^i"]&H_M^{s+2i(p-1),t+i(p-1)}(X;\F_p)\arrow[d,"\cl"]\\
			H^s(t^{\C}(X);\F_p)\arrow[r,"\oP^i"]&H^{s+2i(p-1)}(t^{\C}(X);\F_p).
		\end{tikzcd}
	\end{equation}
	
	\subsection{Totaro's Chow rings of classifying spaces}\label{ssec:Chow of BG} It is well known (Preface of \cite{mazza2011lecture}) that for a smooth scheme $X$ over $k$ we have
	\begin{equation}\label{eq:Mot-Chow-Sm}
		H_M^{2t,t}(X)=\CH^t(X).
	\end{equation}
	This may extend to $X=BG$, in which case $\CH^*(BG)$ is the Chow ring of $BG$ in the sense of \cite{totaro1999chow} and \cite{edidin1998equivariant}. The definition requires some prerequisite as follows.
	\begin{lemma}[Eddidin-Graham, Lemma 9, \cite{edidin1998equivariant}]\label{lem:Eddidin}
		Let $G$ be an algebraic group. For any $i>0$, there is a representation $V$ of $G$ and an open set $U\subset V$ such that $V-U$ has codimension more than $i$ and such that a principal bundle quotient $U\to U/G$ exists in the category of schemes.
	\end{lemma}
	
	\begin{theorem}[Totaro, Theorem 1.1, \cite{totaro1999chow}]\label{thm:Rep-quotient}
		Let $G$ be a linear algebraic group over a field $k$. Let $V $be any representation of $G$ over $k$ such that $G$ acts freely outside a $G$-invariant closed subset $S\subset V$ of codimension $\geq s$. Suppose that the geometric quotient $(V-S)/G$ (in the sense of \cite{mumford1994geometric}) exists as a variety over $k$. Then the ring $\CH^*((V-S)/G)$, restricted to degrees less than $s$, is independent (in a canonical way) of the representation $V$ and the closed subset $S$.
	\end{theorem}
	Now we may present the definition of the Chow ring of a classifying space of an algebraic group.
	\begin{definition}[Totaro, Definition 1.2, \cite{totaro1999chow}]\label{def:CH-BG}
		For a linear algebraic group $G$ over a field $k$, define $\CH^i(BG)$ to be the group $\CH^i((V-S)/G)$ for any $(V,S)$ as in Theorem \ref{thm:Rep-quotient} such that $S$ has codimension greater than $i$ in $V$.
	\end{definition}
	
	The existence of the co-complete category $\Motk$ gives the colimit construction above on the level of (homotopy types) of motivic spaces, which is called the \emph{geometric classifying space} of $G$ and is denoted by $BG$ (4.2, \cite{morel19991}). More precisely, for any base field $k$, consider $\mathbf{HMot}^k_{\bullet}$, the pointed motivic homotopy category over $k$. For a faithful representation $G\times\mathbf{A}^m\to\mathbf{A}^m$, and the associated diagonal representations $G\times\mathbf{A}^{im}\to\mathbf{A}^{im}$. Let $U_i$ be the maximal open sub-scheme of $\mathbf{A}^{im}$ on which $G$ acts freely, and the geometric quotient
	\begin{equation}\label{eq:quotientVi}
		V_i:=U_i/G
	\end{equation}
	exists as a smooth scheme (Lemma \ref{lem:Eddidin}). Then we have a chain of morphisms $\cdots\to V_i\to V_{i+1}\to\cdots$ such that its colimit in $\mathbf{HMot}^k_{\bullet}$ depends on $G$ and is independent of any choice involved.
	
	
	For $G$, $V$, and $V'=V-S$ in Theorem \ref{thm:Rep-quotient}, let $U:=V'/G$. Then we have a principal $G$-bundle $V'\to U$. Its geometric realization is a principal $G(\C)$-bundle $V'(\C)\to U(\C)$. taking homotopy colimits, we have $t^{\C}(BG)\cong B(G(\C))$, as well as the cycle class map
	\begin{equation}\label{eq:cl-BG}
		\begin{cases}
			\cl: H_M^{s,t}(BG;R)\to H^s(BG(\C);R),\\
			\cl:\CH^t(BG;R)\to H^{2t}(BG(\C);R).
		\end{cases}
	\end{equation}
	If there is a compactification $\bar G(\C)$ of the Lie group $G(\C)$, we may write
	\begin{equation*}
		\begin{cases}
			\cl: H_M^{s,t}(BG;R)\to H^s(B\bar{G}(\C);R),\\
			\cl:\CH^t(BG;R)\to H^{2t}(B\bar{G}(\C);R)
		\end{cases}
	\end{equation*}
	instead of \eqref{eq:cl-BG}.
	
	To describe the universal property of $BG$, we need to work in the category $\mathbf{HMot}^k_{Nis}$, the homotopy category of motivic spaces with respect to the localization with respect to the Nisnevich topology.
	
	With some general model-categorical construction (\cite{morel19991}, Chapter 4),  we obtain an isomorphism of functors 
	\begin{equation}\label{eq:etale-G}
		\He^1(-;G)\cong \opn{Hom}_{\mathbf{HMot}^k_{Nis}}(-, BG).
	\end{equation}

	\section{The classes $\rho_{p,k}$ and $y_{p,k}$}\label{sec:rho}
	Let $p$ be an odd prime, and $n$ a positive integer divisible by $p$. In this section we recall the $p$-torsion classes $y_{p,k}\in H^{2p^{k+1}+2}(BPU_n)$, and construct $p$-torsion classes $\rho_{p,k}\in\CH^{p^{p+1}+1}(BPGL_n)$ satisfying $\cl(\rho_{p,k})=y_{p,k}$.
	
	In \cite{gu2019cohomology} and \cite{gu2019some}, $\Gu$ considered the following construction. By the definition of $PU_n$, we have a short exact sequence
	\begin{equation}\label{eq:SES-Lie}
		1\to S^1\to U_n\to PU_n\to 1,
	\end{equation}
	which yields a homotopy fiber sequence
	\[BS^1\to BU_n\to BPU_n.\]
	As $BS^1$ is of the homotopy type of the Eilenberg-Mac Lane space $K(\Z,2)\simeq \Omega K(\Z,3)$, we have the Puppe sequence which extends the above to another homotopy fiber sequence
	\begin{equation}\label{eq:main-fib-seq}
		BU_n\to BPU_n\xrightarrow{\chi}K(\Z,3).
	\end{equation}
	
	Alternatively, the map $\chi$ may be constructed as follows. Consider the short exact sequence
	\begin{equation*}
		1\to\mu_n\to SU_n\to  PU_n\to 1,
	\end{equation*}
	where $\mu_n$ is the cyclic group of complex $n$th roots of unity. The sequence yields a Bockstein homomorphism
	\[\delta_{PU_n}: H^1(-;PU_n)\to H^2(-;\Z/n).\]
	\begin{lemma}\label{lem:chi-alt}
		The map $\chi:BPU_n\to K(\Z,3)$ represents the following composition:
		\begin{equation*}
			\opn{Hom}_{\Top}(-,BPU_n)\cong H^1(-;PU_n)\xrightarrow{\delta_{PU_n}}H^2(-;\Z/n)\xrightarrow{\delta}H^3(-;\Z).
		\end{equation*}
	\end{lemma}
	The proof is a routine check.
	
	The classes $y_{p,k}$ are defined by means of the map $\chi$ and the cohomology of $K(\Z,3)$. In general, the cohomology of the Eilenberg-Mac Lane space $K(A,n)$ for $A$ a finitely generated abelian group can be deduced from \cite{cartan19551955}.  The integral cohomology of $K(\Z,3)$ is described in \cite{gu2019cohomology} in terms of Steenrod reduced power operations, resembling the description of the mod $p$ cohomology of $K(A,n)$ by Tamanoi \cite{tamanoi1999subalgebras}. Instead of repeating the above results, we only presents some particular cohomology classes.
	
	Let
	\[\delta: H^*(-;\F_p)\to H^{*+1}(-)\]
	be the Bockstein homomorphism,
	\[\beta:H^*(-;\F_p)\to H^{*+1}(-;\F_p)\]
	the mod $p$ reduction of $\delta$, and $\oP^i$ the $i$th Steenrod reduced power operation.
	\begin{proposition}\label{pro:K(Z,3)}
		Let $x_1\in H^3(K(\Z,3))$ be the fundamental class of $K(\Z,3)$, i.e., the class represented by the identity morphism of $K(\Z,3)$. For $k\geq0$, there are nontrivial $p$-torsion cohomology classes
		\[y_{p,k}:=\delta\oP^{p^k}\oP^{p^{k-1}}\cdots\oP^1(\bar{x}_1)\in H^{2p^{k+1}+2}(K(\Z,3)),\]
		where $\bar{x}_1$ denote the mod $p$ reduction of $x_1$.
	\end{proposition}
	In \cite{gu2019some}, $\Gu$ shows the following
	\begin{proposition}[Theorem 1.1, \cite{gu2019some}]\label{pro:y-p-k-nontrivial}
		For $p\mid n$ and $k\geq 0$, the classes $\chi^*(y_{p,k})\in H^{2p^{k+1}+2}(BPU_n)$ are nontrivial.
	\end{proposition}
	For simplicity, we omit the notation $\chi^*$ and write $x_1\in H^3(BPU_n)$ and $y_{p,k}\in H^{2p^{k+1}+2}(BPU_n)$ instead.
	
	We proceed to construct a motivic counterpart of $x_1$. 
	Consider the short exact sequence of algebraic groups
	\begin{equation*}
		1\to\mu_n\to SL_n\to  PGL_n\to 1,
	\end{equation*}
	which induces a morphism in $\mathbf{HMot}^{\C}_{Nis}$:
	\begin{equation}\label{eq:con-zeta-1}
		\delta_{PGL_n}:BPGL_n\to B^2\mu_n.
	\end{equation}
	
	On the other hand, we have
	\begin{equation}\label{eq:con-zeta-2}
		\begin{split}
			&\He^2(-;\mu_n) \\
			\cong & \He^2(-;\mu_n^{\otimes 2})\ \ (\textrm{$\C$ containing a primitive $n$th root of unity})\\
			\cong & H_M^{2,2}(-;\Z/n) \ \ (\textrm{Theorem \ref{thm:BL}})\\
			\cong & \opn{Hom}_{\Mot}(-, K(\Z/n(2),2)).
		\end{split}
	\end{equation}
	As shown in \cite{morel19991}, Chapter 4, for any Nisnevich sheaf of groups $G$ over $\Smk$, there is an isomorphism
	\begin{equation}\label{eq:con-zeta-5}
		\opn{Hom}_{\mathbf{Mot}^k_{Nis}}(-,BG) \cong \He^1(-,G).
	\end{equation}
	Let $G$ be the sheaf $\He^1(-,\mu_n)$. Then we have
	\begin{equation}\label{eq:con-zeta-6}
		\He^2(-,\mu_n) \cong \opn{Hom}_{\mathbf{Mot}^k_{Nis}}(-,B^2\mu_n).
	\end{equation}
	Now let $k=\C$. By Proposition \ref{pro:mu_ninv}, $B^2\mu_n$ is $\mathbf{A}^1$-local. Therefore, \eqref{eq:con-zeta-6} is further improved into 
	\begin{equation}\label{eq:con-zeta-3}
		\He^2(-;\mu_n)\cong \opn{Hom}_{\Mot}(-, B^2\mu_n),	
	\end{equation}
	and by \eqref{eq:con-zeta-2} and \eqref{eq:con-zeta-3} we have a canonical isomorphism
	\begin{equation}\label{eq:con-zeta-4}
		B^2\mu_n\cong K(\Z/n(2),2)
	\end{equation}
	in the category $\Mot$. Combining \eqref{eq:con-zeta-1} and \eqref{eq:con-zeta-2}, we have $\delta_{PGL_n}$ of the form
	\begin{equation*}
		\delta_{PGL_n}: BPGL_n\to K(\Z/n(2),2).
	\end{equation*}
	We then take the following compositon, which is a  morphism in $\Mot$ denoted by
	\begin{equation}\label{eq:def-zeta}
		\chi_M: BPGL_n\to K(\Z/n(2),2)\xrightarrow{\delta} K(\Z(2),3)
	\end{equation}
	where $\delta$ is the Bockstein homomorphism. Let $\zeta_1\in H_M^{3,2}(BPGL_n)$ be the class represented by $\chi_M$. Then $\zeta_1$ is an $n$-torsion class. It is the desired motivic counterpart of $x_1$, in the sense of the following
	\begin{lemma}\label{lem:tC-zeta}
		\begin{equation*}
			\cl(\zeta_1)=x_1\in H^3(BPU_n).
		\end{equation*}
	\end{lemma}
	
	\begin{proof}
		This follows immediately from Proposition \ref{pro:factor-cl} and Lemma \ref{lem:chi-alt}.
	\end{proof}
	In what follows, we let overhead bars indicate mod $p$ reductions of integral (motivic and singular) cohomology classes.
	\begin{definition}\label{def:rho-p-k}
		For $p$ an odd prime, $p\mid n$, and $k\geq 0$, we define $p$-torsion classes
		\begin{equation*}
			\begin{split}
				\rho_{p,k}:=\delta\oP^{p^k}\oP^{p^{k-1}}&\cdots\oP^p\oP^1(\bar{\zeta}_1)\\
				&\in H_M^{2p^{k+1}+2,p^{k+1}+1}(BPGL_n)=\CH^{p^{k+1}+1}(BPGL_n).
			\end{split}
		\end{equation*}
	\end{definition}
	The classes $\rho_{p,k}$ satisfy the properties given in Theorem \ref{thm:main}:
	\begin{proposition}
		For $p\mid n$, the classes $\rho_{p,k}\in \CH^{p^{k+1}+1}(BPGL_n)$ satisfy
		\begin{equation*}
			\cl(\rho_{p,k})=y_{p,k}.
		\end{equation*}
	\end{proposition}
	\begin{proof}
		This follows immediately from Lemma \ref{lem:tC-zeta} and the functorial property of $\cl$, and the compatibility of the Steenrod reduced power operations and $\cl$:
		\begin{equation*}
			\begin{tikzcd}
				H_M^{3,2}(BPGL_n;\F_p)\arrow[r]\arrow[d,"\cl"]& H_M^{2p^{k+1}+2,p^{k+1}+1}(BPGL_n)\arrow[d,"\cl"]\\
				H^3(BPU_n;\F_p)\arrow[r]& H^{2p^{k+1}+2}(BPU_n),
			\end{tikzcd}
		\end{equation*}
		where the horizontal arrows are the operations $\delta\oP^{p^k}\oP^{p^{k-1}}\cdots\oP^1$.
	\end{proof}

	\section{On the extraspecial $p$-groups $p_{+}^{1+2r}$}\label{sec:extraspectial}
	For an odd prime number $p$, a finite $p$-group $G$ is called an \textit{extraspecial $p$-group} if its center $Z(G)$ is cyclic of order $p$, and the quotient $G/Z(G)$ is a nontrivial elementary abelian $p$-group, i.e., an abelian group in which every nontrivial element is of order $p$. A particular type of extraspecial $p$-groups play an important role in the construction of  non-toral $p$-elementary subgroups of $PU_{p^r}$.
	
	The complete classification of extraspecial $p$-groups is known, by a theorem of P. Hall (Theorem 5.4.9, \cite{gorenstein2007finite}). In this section, we concern ourselves with only one type of extraspecial $p$-groups for each odd prime $p$. The main result of this section is Lemma \ref{lem:H2-of-p}.
	
	The cohomology of the extraspecial $p$-groups are studied in depth by Tezuka and Yagita \cite{tezuka1983varieties} and Benson and Carlson \cite{benson1992cohomology}. In this parer we merely need a partial result, which we deduce independently, for the sake of completeness.
	
	Throughout the rest of this paper, we denote by $Z(G)$ the center of a group $G$.
	
	The orders of extraspectial $p$-groups are of the form $p^{1+2r}$ for $r>0$, and conversely, for each $r>0$ we have two extra special $p$-groups of order $p^{1+2r}$, one of which is denoted by $p_+^{1+2r}$. We present $p_+^3$ in terms of generators and relations:
	\begin{equation}\label{eq:p3+}
		p_{+}^3:=\langle z,e_1,f_1\mid  e_1z=ze_1,\ f_1z=zf_1,\ e_1f_1=zf_1e_1\rangle.
	\end{equation}
	It follows that $Z(p_{+}^3)$ is the cyclic group $\Z/p$ generated by $z$, and the quotient group $p_{+}^3/Z(p_{+}^3)$ is isomorphic to $(\Z/p)^2$, which is commutative. To study the groups $p_{+}^{1+2r}$ for $r>1$, we recall the following
	\begin{definition}\label{def:central-product}
		Let $G_1, G_2$ be groups such that there is an isomorphism $\phi:Z(G_1)\to Z(G_2)$. The \textit{central product} of $G_1$ and $G_2$ with respect to $\phi$ is
		\[G_1*_{\phi}G_2 := (G_1\times G_2)/\{(z,\phi(z))\mid  z\in Z(G_1)\}.\]
		We often omit the subscript $\phi$ when it is clear from the context. In particular, we write $G*G$ in the case that $\phi$ is the identity on $Z(G)$.
	\end{definition}
	\begin{remark}
		The central product is associative and we feel free to write $G_1*_{\phi_1}G_2*_{\phi_2}\cdots *_{\phi_{r-1}}G_r$, and in particular $G*G*\cdots *G$.
	\end{remark}
	\begin{definition}\label{def:p1+2r+}
		We define the group
		\[p_{+}^{1+2r}:=p_{+}^3*\cdots *p_{+}^3 \ (r\textrm{-fold central product}).\]
	\end{definition}
	The following is well known to group theorists, and its proof is a straightforward computation.
	\begin{proposition}\label{pro:p1+2r+}
		The group $p_{+}^{1+2r}$ is an extraspecial $p$-group of order $p^{1+2r}$, with the following presentation in terms of generators and relations:
		\begin{itemize}
			\item a set of generators $z,e_i,f_i$ for $1\leq i\leq r$, and
			\item relations
			\[[e_i,z],\ [f_i,z],\ [e_i,e_j],\ [f_i,f_j],\ [e_i,f_i]z^{-1}, \textrm{ and }[e_i,f_j]\textrm{ for }i\neq j,\]
			where $[a,b]$ denotes the commutator $aba^{-1}b^{-1}$.
		\end{itemize}
	\end{proposition}
	
	In the rest of this paper, we use both $V$ and $\Z/p$ as notations for the cyclic group of order $p$. We use the former if we consider its classifying space, and the latter if we regard it as subgroup or quotient of a (co)homology group or a Chow ring. In particular, we denote by $V^k$ the $k$-fold direct sum of $V^k$.
	
	The following is an immediate consequence of Proposition \ref{pro:p1+2r+}
	\begin{corollary}\label{cor:center-SES}
		Let $V^{2r}=(\Z/p)^{2r}$ be the Cartesian product of cyclic groups of order $p$, with a basis $e_1,\cdots,e_r,\ f_1,\cdots,f_r$. There is a short exact sequence of groups
		\[1\to V \to p_{+}^{1+2r} \to V^{2r}\to 1,\]
		where $\Z/p$ maps onto $Z(p_{+}^{1+2r})$.
	\end{corollary}
	
	\begin{lemma}\label{lem:H2-of-p}
		The second cohomology group of $Bp_{+}^{1+2r}$ is
		\[H^2(Bp_{+}^{1+2r}) \cong (\Z/p)^{\oplus 2r}.\]
	\end{lemma}
	\begin{proof}
		Since $p_+^{1+2r}$ is a finite group, the cohomology and homology of $Bp_+^{1+2r}$ are torsion abelian groups. Therefore, by the universal coefficient theorem we have
		\[H^2(Bp_+^{1+2r}) \cong H_1(Bp_+^{1+2r};\Z) \cong p_+^{1+2r}/ \opn{Comm}(p_+^{1+2r}),\]
		where $\opn{Comm}(p_+^{1+2r})$ is the subgroup of $p_+^{1+2r}$ generated by commutators. The desired result then follows from
		\[p_+^{1+2r}/ \opn{Comm}(p_+^{1+2r}) \cong V^{2r},\]
		an elementary computation based on Proposition \ref{pro:p1+2r+}.
	\end{proof}
	
	\section{A non-toral $p$-elementary subgroup of $PU_{p^r}$}\label{sec:p-elementary}
	In this section we prove that the ring homomorphisms \eqref{eq:rho-p-k} and \eqref{eq:y-p-k} in Theorem \ref{thm:main} are injective, by studying the cohomology of a  $p$-elementary subgroup of $PU_{p^r}$. Since we have the cycle class map $\cl:\CH^*(BPGL_n)\to H^*(BPU_n)$ with $\cl(\rho_{p,k})=y_{p,k}$, the injectivity of \eqref{eq:rho-p-k} follows from that of \eqref{eq:y-p-k}. Hence, we will focus on the proof of \eqref{eq:y-p-k} in this section.
	
	The non-toral $p$-elementary subgroups of $PU_n$ and their normalizers are studied by Griess \cite{griess1991elementary} (Table \RomanNumeralCaps{2}), where a systematic investigation of elementary $p$-subgroups of algebraic groups is carried out. Andersen, Grodal, M{\o}ller, and Viruel \cite{andersen2008classification} present a more detailed discussion. For the purpose of this section, it suffices to consider the case $n=p^r$ for $p$ an odd prime.
	
	In the special case $r=1$, much of the constructions presented in this section appears in various works such as  \cite{vistoli2007cohomology}, \cite{kameko2015integral}, and \cite{kono1993brown}.
	
	We present the $p$-elementary subgroups of $PU_{p^r}$ as follows. First we construct monomorphisms of Lie groups $\bar{\theta}:p_{+}^{1+2r}\hookrightarrow U_{p^r}$, where $p_{+}^{1+2r}$ is the extraspecial $p$-group studied in Section \ref{sec:extraspectial}. Passing to quotients over centers we obtain monomorphisms of the form $\theta:V^{2r}\to PU_{p^r}$, where $V^{2r}=(\Z/p)^{\oplus 2r}$ as in Section \ref{sec:extraspectial}.
	
	We proceed to present the monomorphisms $\bar{\theta}:p_{+}^{1+2r}\hookrightarrow U_{p^r}$. First we consider $r=1$, in which case we have \eqref{eq:p3+}:
	\begin{equation*}
		p_{+}^3:=\langle z,e_1,f_1\mid  e_1z=ze_1,\ f_1z=zf_1,\ e_1f_1=zf_1e_1\rangle.
	\end{equation*}
	We define $\bar{\theta}: p_{+}^3\to U_p$ by
	\begin{equation*}
		\bar{\theta}(z)=e^{\frac{2\pi i}{p}}I_p,\
		\bar{\theta}(e_1)=
		\begin{pmatrix}
			e^{\frac{2\pi i}{p}} & & & \\
			& \ddots & & \\
			& & e^{\frac{2\pi i(p-1)}{p}} & \\
			& & & 1
		\end{pmatrix},\
		\bar{\theta}(f_1)=
		\begin{pmatrix}
			& 1\\
			I_{p-1} &
		\end{pmatrix}.
	\end{equation*}
	It is straightforward to check that the above indeed gives a monomorphism of Lie groups.

	Taking $r$-fold direct produces, we obtain a homomorphism
	\[\bar{\theta}^{\times r}: (p_+^3)^{\times r}\hookrightarrow U_p^{\times r}\hookrightarrow U_{p^r},\]
	where the inclusion $U_p^{\times r}\hookrightarrow U_{p^r}$ is given by the canonical action of $U_p^{\times r}$ on the $r$-fold tensor product of $\mathbb{C}^p$ with the canonical Hermitian inner product.
	For $z\in Z(p_+^3)$, let
	\begin{equation*}
		z^{(i)}:=(1,\cdots, \stackrel{i\textrm{th}}{z}, \cdots,1)\in (p_+^3)^{\times r}.
	\end{equation*}
	Notice that the element $\bar{\theta}^{\times r}(z^{(i)})$ is independent of $i$, and the above homomorphism factors through the $r$-fold central product and we have a homomorphism
	\begin{equation*}
		p_+^{1+2r}\cong (p_+^3)^{* r}\hookrightarrow U_{p^r}
	\end{equation*}
	which is also denoted by $\bar{\theta}$. Taking the quotient group over the centers on both sides, we obtain a monomorphism
	\begin{equation}\label{eq:theta}
		\theta:\ V^{2r}\hookrightarrow PU_{p^r}.
	\end{equation}
	
	Let $N(V^{2r})$ be the normalizer of $V^{2r}$ in $PU_{p^r}$, and let $W=N(V^{2r})/V^{2r}$. Then the group $W$ acts upon the cohomology ring $H^*(BV^{2r})$ in such a way that the
	restriction homomorphism $\theta^*:H^*(BPU_{p^r})\to H^*(BV^{2r})$ has image in $H^*(BV^{2r})^W$, the subring of $H^*(BV^{2r})$ of $W$-invariants. It is therefore important to study the group $W$ and its action on $H^*(BV^{2r})$, for which we introduce a symplectic bilinear form on $V^{2r}$.
	
	Recall the generators $z, e_i,f_i$, $1\leq i \leq r$ of $p_+^{1+2r}$ as given in Proposition \ref{pro:p1+2r+}. The quotient group $V^{2r}=p_+^{1+2r}/Z(p_+^{1+2r})$ is generated by $e_i,f_i$. In the obvious way, we regard $V^{2r}$ as a $\F_p$-vector space of dimension $2r$ with a basis
	\begin{equation}\label{eq:e-f-basis}
		e_1\cdots,e_r,\ f_1,\cdots,f_r.
	\end{equation}
	Let $\langle-, -\rangle$ be a simplectic bilinear form on $V^{2r}$, such that its matrix associated to the basis \eqref{eq:e-f-basis} is
	\begin{equation}\label{eq:Omega}
		\Omega=
		\begin{pmatrix}
			0 & I_r\\
			-I_r & 0
		\end{pmatrix}.
	\end{equation}
	The following is a special case of Theorem 8.5 of \cite{andersen2008classification}.
	\begin{proposition}[Andersen-Grodal-M{\o}ller-Viruel, Theorem 8.5, \cite{andersen2008classification}]\label{pro:non-toral-inv}
		The normalizer of $V^{2r}$ in $PU_{p^r}$ is $\Sp_r$, the symplectic group over $\F_p$ of order $2r$, which acts on $V^{2r}$ with respect to the symplectic bilinear form $\langle-,- \rangle$.
	\end{proposition}
	Consider the cohomology algebra
	\begin{equation}\label{eq:H-BV-2r-p}
		H^*(BV^{2r};\F_p)=\Lambda_{\F/p}[a_1,\cdots,a_r,\ b_1,\cdots,b_r]\otimes\F_p[\bar{\xi}_1\cdots,\bar{\xi}_r,\ \bar{\eta}_1,\cdots,\bar{\eta}_r].
	\end{equation}
	Here we have $a_i,b_i\in H^1(BV^{2r};\F_p)$, and $\bar{\xi}_i,\bar{\eta}_i$ are respectively the mod $p$ reductions of the integral cohomology classes $\xi_i,\eta_i\in H^2(BV^{2r})$ which satisfy
	\begin{equation}\label{eq:def-xi-eta}
		\xi_i=\delta(a_i),\ \eta_i=\delta(b_i),
	\end{equation}
	where $\delta:H^*(-;\F_p)\to H^{*+1}(-)$ denotes the Bockstein homomorphism. In other words, we have $\bar{\xi}_i=\beta(a_i)$ and $\bar{\eta}_i=\beta(b_i)$ where $\beta$ is the mod $p$ reduction of $\delta$. By Proposition \ref{pro:non-toral-inv} we have
	\begin{corollary}\label{cor:non-toral-inv}
		For a suitable choice of $a_i,b_i$, $1\leq i\leq r$ as above, and a symplectic bilinear form $\langle-,- \rangle$ on the $\F_p$-vector space $H^1(BV^{2r};\F_p)$ given by the matrix $\Omega$ with respect to the basis $a_1\cdots,a_r,\ b_1,\cdots,b_r$, the $\Sp_r$-actions on $H^*(BV^{2r};\F_p)$ and $H^*(BV^{2r})$ are described as follows. Suppose $g\in\Sp_r$.
		\begin{itemize}
			\item It acts tautologically as the symplectic transformations on the $\F_p$-vector space $H^1(BV^{2r};\F_p)$ with respect to the symplectic bilinear form $\langle-,- \rangle$.
			\item For $g\in\Sp_r$ and $a\in H^1(BV^{2r};\F_p)$, we have $g\beta(a)=\beta(ga)$.
			\item For $a,b\in H^*(BV^{2r};\F_p)$, we have $g(ab)=(ga)(gb)$.
			\item For any $\xi\in H^k(BV^{2r})$, $k>0$, there is a unique $a\in  H^{k-1}(BV^{2r};\F_p)$ satisfying $\xi=\delta(a)$, and we have $g\xi=\delta(ga)$.
		\end{itemize}
		In particular, the Bockstein homomorphism $\delta$ is $\Sp_r$-equivariant.
	\end{corollary}
	\begin{lemma}\label{lem:sympletic-inv}
		Let
		\[\Lambda^*=\Lambda_{\F_p}[a_1,\cdots,a_r,\ b_1,\cdots,b_r]\]
		be the graded exterior $\F_p$-algebra generated by $a_1,\cdots,a_r,b_1,\cdots,b_r$, each of which is of degree $1$, regarded as an  subalgebra of $H^*(BV^{2r};\F_p)$ in the sense of \eqref{eq:H-BV-2r-p}. Then the $\Sp_r$-action on $H^*(BV^{2r};\F_p)$ in Corollary \ref{cor:non-toral-inv} restricts to $\Lambda^*$, and the $\Sp_r$-invariant $\F_p$-subspace of $\Lambda^2$ is generated by $\sum_{i=1}^r a_ib_i$.
	\end{lemma}
	\begin{proof}
		It is straightforward to check that the $\Sp_r$-action on $H^*(BV^{2r};\F_p)$ in Corollary \ref{cor:non-toral-inv} restricts to $\Lambda^*$.
		
		An arbitrary element in $\Lambda^2$ may be written as
		\[w=\sum_{i,j}(r_{ij}a_ia_j + s_{ij}a_ib_j + t_{ij}b_ib_j),\]
		for $r_{ij},s_{ij},t_{ij}\in\F_p$, or more conveniently
		\begin{equation}\label{eq:RST-matrix}
			w=
			\begin{pmatrix}
				\underline{a} & \underline{b}
			\end{pmatrix}
			\begin{pmatrix}
				R & S\\
				0 & T
			\end{pmatrix}
			\begin{pmatrix}
				\underline{a}^t \\ \underline{b}^t
			\end{pmatrix}
		\end{equation}
		where we have
		\begin{equation*}
			\underline{a}=\begin{pmatrix}
				a_1\cdots a_n
			\end{pmatrix},\
			\underline{b}=\begin{pmatrix}
				b_1\cdots b_n
			\end{pmatrix}
		\end{equation*}
		and
		\begin{equation*}
			R=(r_{ij}),\ S=(s_{ij}), \ T=(t_{ij})\in\F_p^{r\times r}.
		\end{equation*}
		Hence, the class $w$ is $\Sp_r$-invariant if and only if for any $P\in\Sp_r$ we have
		\begin{equation}\label{eq:Sp-inv-matrix}
			\begin{pmatrix}
				R & S\\
				0 & T
			\end{pmatrix}=
			P
			\begin{pmatrix}
				R & S\\
				0 & T
			\end{pmatrix}
			P^t.
		\end{equation}
		For \eqref{eq:Sp-inv-matrix} to hold for all
		\begin{equation*}
			P\in\{\begin{pmatrix}
				A & 0\\
				0 & (A^t)^{-1}
			\end{pmatrix}
			\mid A\in GL_r(\F_p)\}\subset\Sp_r,
		\end{equation*}
		it is necessary that we have $R=T=0$ and $S=sI_r$ for some $s\in\F_p$, which are easily verified also as a sufficient condition for \eqref{eq:Sp-inv-matrix}. Therefore we have
		\begin{equation*}
			w=
			\begin{pmatrix}
				\underline{a} & \underline{b}
			\end{pmatrix}
			\begin{pmatrix}
				0 & sI_r\\
				0 & 0
			\end{pmatrix}
			\begin{pmatrix}
				\underline{a}^t \\ \underline{b}^t
			\end{pmatrix}
			=s\sum_{i=1}^r a_ib_i.
		\end{equation*}
	\end{proof}
	\begin{proposition}\label{pro:H3-inv}
		We have the invariant subgroup $H^3(BV^{2r})^{\Sp_r}\cong\Z/p$, which is generated by the class $\delta(\sum_{i=1}^r a_ib_i)$.
	\end{proposition}
	\begin{proof}
		The short exact sequence
		\[0\to\Z\xrightarrow{\times p}\Z\to\F_p\to 0\]
		induces a long exact sequence
		\[\cdots\to H^k(BV^{2r})\xrightarrow{\times p}H^k(BV^{2r})\xrightarrow{q} H^k(BV^{2r};\F_p)\xrightarrow{\delta}H^{k+1}(BV^{2r})\to\cdots.\]
		Since the groups $H^k(BV^{2r})$ are $p$-torsion for $k>0$, the long exact sequence breaks down to short exact sequences
		\[0\to H^k(BV^{2r})\xrightarrow{q} H^k(BV^{2r};\F_p)\xrightarrow{\delta}H^{k+1}(BV^{2r})\to 0\]
		for $k>0$, and in particular, we have an $\Sp_r$-equivariant isomorphism induced by $\delta$:
		\begin{equation}\label{eq:H3-iso}
			H^2(BV^{2r};\F_p)/q(H^2(BV^{2r}))\xrightarrow{\cong}H^3(BV^{2r}),
		\end{equation}
		where the left hand side is an $\F_p$-vector space with a basis consisting of the conjugate classes of
		\begin{equation*}
			\begin{cases}
				a_ia_j,\ b_ib_j,\ 1\leq i<j\leq r,\\
				a_ib_j,\ 1\leq i,j\leq r.
			\end{cases}
		\end{equation*}
		The proposition now follows from Lemma \ref{lem:sympletic-inv}.
	\end{proof}
	
	\begin{proposition}\label{pro:comparison}
		The homomorphism
		\[B\theta^*:H^3(BPU_{p^r})\to H^3(BV^{2r})^{\Sp_r}\]
		is surjective. In other words, we have
		\[B\theta^*(x_1)=\lambda\delta(\sum_{i=1}^ra_ib_i),\]
		for some $\lambda\in\Z$, $p\nmid\lambda$ .
	\end{proposition}
	\begin{proof}
		By Corollary \ref{cor:center-SES} we have the following commutative diagram:
		\begin{equation*}
			\begin{tikzcd}
				\Z/p\arrow[r]\arrow[d]& p_{+}^{1+2r}\arrow[r]\arrow[d]& V^{2r}\arrow[d,"\theta"]\\
				S^1\arrow[r]& U_{p^r}\arrow[r]& PU_{p^r},
			\end{tikzcd}
		\end{equation*}
		of which both rows are short exact sequences of groups. Hence we have a commutative diagram of fiber sequences
		\begin{equation}\label{eq:UV-fib-seq-diag}
			\begin{tikzcd}
				Bp_{+}^{1+2r}\arrow[r]\arrow[d]& BV^{2r}\arrow[r,"\upsilon"]\arrow[d,"B\theta"]& K(\Z/p,2)\arrow[d,"D"]\\
				BU_{p^r}\arrow[r]& BPU_{p^r}\arrow[r,"\chi"]& K(\Z,3),
			\end{tikzcd}
		\end{equation}
		where $D$ is the map representing the Bockstein homomorphism
		\[H^2(-;\F_p)\to H^3(-).\]
		Now we have the following commutative diagram:
		\begin{equation}\label{eq:SSBV-BPU}
			\begin{tikzcd}
				H^*(K(\Z,3))\arrow[r,"\chi^*"]\arrow[d,"D^*"]& H^*(BPU_n)\arrow[d,"B\theta^*"]\\
				H^*(K(\Z/p,2))\arrow[r,"\upsilon^*"]& H^*(BV^{2r}).
			\end{tikzcd}
		\end{equation}
		Let $(^VE_*^{*,*}, {^Vd}_*^{*,*})$ be the integral cohomological Serre spectral sequence associated to the second row of \eqref{eq:UV-fib-seq-diag}:
		\begin{equation*}
			\begin{split}
				& ^VE_2^{s,t}=H^s(K(\Z/p,2);H^t(Bp_{+}^{1+2r}))\Rightarrow H^{s+t}(BV^{2r}),\\
				& ^Vd_2^{s,t}: {^VE}_2^{s,t}\to {^VE}_2^{s+r,t-r+1}.
			\end{split}
		\end{equation*}
		By Lemma \ref{lem:H2-of-p}, we have
		\begin{equation}\label{eq:lem-H2-of-p}
			^VE_2^{0,2} = H^2(Bp_{+}^{1+2r}) \cong (\Z/p)^{\oplus 2r}.
		\end{equation}
		On the other hand, by the universal coefficient theorem we have
		\begin{equation}\label{eq:H2-of-BV2r}
			H^2(BV^{2r})\cong (\Z/p)^{\oplus 2r}.
		\end{equation}
		Observe that the only nontrivial entry of $^VE_2^{*,*}$ of total degree $2$ is $^VE_2^{0,2}$, i.e., we have
		\begin{equation}\label{eq:VE-infty-2}
			^VE_{\infty}^{0,2}\cong H^2(BV^{2r}).
		\end{equation}
		By \eqref{eq:lem-H2-of-p}, \eqref{eq:H2-of-BV2r}, and \eqref{eq:VE-infty-2}, we have
		\begin{equation}\label{eq:VE-infty-0-2}
			^VE_{\infty}^{0,2}= {^VE}_2^{0,2}.
		\end{equation}
		It follows from \eqref{eq:VE-infty-0-2} that there is no nontrivial differential landing on $^VE_*^{3,0}$. Therefore, we have
		\begin{equation}\label{eq:VE-infty-3-0}
			^VE_{\infty}^{3,0}={^VE}_2^{3,0}=H^3(K(\Z/p,2))\cong\Z/p.
		\end{equation}
		In other words, we have a short exact sequence
		\begin{equation}\label{eq:upsilon-inj}
			0\to H^3(K(\Z/p,2))\xrightarrow{\upsilon^*} H^3(BV^{2r}).
		\end{equation}
		
		On the other hand, we have
		\begin{equation}\label{eq:Im-chi}
			H^3(BPU_{p^r})=\opn{Im}\{\chi^*:H^3(K(\Z,3))\to H^3(BPU_{p^r})\}\cong\Z/p^{r},
		\end{equation}
		which follows by studying the differentials of the Serre spectral sequence
		\[^UE_2^{s,t}=H^s(K(\Z,3);H^t(BU_{p^r}))\Rightarrow H^{s+t}(BPU_{p^r}).\]
		For instance, see Corollary 3.4 of \cite{gu2019cohomology}.
		
		Comparing  \eqref{eq:SSBV-BPU}, \eqref{eq:upsilon-inj}, and \eqref{eq:Im-chi}, we have
		\begin{equation*}
			\begin{split}
				&\opn{Im}\{B\theta^*:H^3(BPU_{p^r})\to H^3(BV^{2r})\}\\
				=&\opn{Im}\{B\upsilon^*:H^3(K(\Z/p,2))\to H^3(BV^{2r})\}\cong\Z/p.
			\end{split}
		\end{equation*}
		Compare the above and Proposition \ref{pro:H3-inv}, and we conclude.
	\end{proof}
	The following is not required for the proof of Theorem \ref{thm:main}, but nontheless interesting.
	\begin{corollary}\label{cor:H2-p1+2r}
		$H^2(Bp_{+}^{1+2r})\cong (\Z/p)^{\oplus 2r}$.
	\end{corollary}
	\begin{proof}
		This follows from Lemma \ref{lem:H2-of-p}, \eqref{eq:H2-of-BV2r}, and \eqref{eq:VE-infty-0-2}.
	\end{proof}
	Recall the classes $\xi_i=\delta(a_i),\eta_i=\delta(b_i)\in H^2(BV^{2r})$.
	\begin{corollary}\label{cor:Btheta}
		There is a $\lambda\in\Z$, $p\nmid\lambda$, satisfying
		\[B\theta^*(y_{p,k})=\lambda\sum_{i=1}^r(\xi_i^{p^{k+1}}\eta_i-\xi_i\eta_i^{p^{k+1}})\]
		for all $k\geq0$.
	\end{corollary}
	\begin{proof}
		This is a computation involving Steenrod reduced power operations. Consider the cohomology algebra
		\[H^*(BV^{2r};\F_p)=\Lambda_{\F/p}[a_1,\cdots,a_r,\ b_1,\cdots,b_r]\otimes\F_p[\bar{\xi}_1\cdots,\bar{\xi}_r,\ \bar{\eta}_1,\cdots,\bar{\eta}_r],\]
		and recall the relations
		\begin{equation*}
			\xi_i=\delta(a_i),\ \eta_i=\delta(b_i).
		\end{equation*}
		We recall the two most relevant of the axioms for the Steenrod reduced power operations:
		\begin{itemize}
			\item Dimension axiom:
			\begin{equation*}
				\oP^i(x)=
				\begin{cases}
					x^p,\ \textrm{for}\ x \textrm{ of cohomological dimension } 2i,\\
					0,\ \textrm{for}\ x \textrm{ of cohomological dimension} <2i.
				\end{cases}
			\end{equation*}
			In particular, for $k\geq0$, we have
			\begin{equation*}
				\begin{split}
					&\oP^{p^k}(\xi_i^{p^k})=\xi_i^{p^{k+1}},\ \oP^{p^k}(\eta_i^{p^k})=\eta_i^{p^{k+1}},\\
					&\oP^j(a_i)=\oP^j(b_i)=0,\ \forall j>0,
				\end{split}
			\end{equation*}
			\item Cartan formula: $\oP^k(x\cdot y)=\sum_{i+j=k}\oP^i(x)\cdot\oP^j(y)$.
		\end{itemize}
		The computation is then carried out as follows:
		\begin{equation*}
			\begin{split}
				B\theta^*(y_{p,k})&=B\theta^*(\delta\oP^{p^k}\oP^{p^{k-1}}\cdots\oP^p\oP^1(\bar{x}_1))\\
				&=\delta\oP^{p^k}\oP^{p^{k-1}}\cdots\oP^p\oP^1(\lambda\cdot\beta(\sum_{i=1}^ra_ib_i))\\
				&=\lambda\cdot\delta\oP^{p^k}\oP^{p^{k-1}}\cdots\oP^p\oP^1[\sum_{i=1}^r(\bar{\xi}_ib_i-a_i\bar{\eta}_i)]\\
				&=\lambda\cdot\delta\oP^{p^k}\oP^{p^{k-1}}\cdots\oP^p[\sum_{i=1}^r(\bar{\xi}_i^pb_i-a_i\bar{\eta}_i^p)]\\
				&=\cdots\\
				&=\lambda\cdot\delta[\sum_{i=1}^r(\bar{\xi}_i^{p^{k+1}}b_i-a_i\bar{\eta}_i^{p^{k+1}})]\\
				&=\lambda\cdot\sum_{i=1}^r(\xi_i^{p^{k+1}}\eta_i-\xi_i\eta_i^{p^{k+1}}).
			\end{split}
		\end{equation*}
	\end{proof}
	\begin{lemma}\label{lem:alg-ind}
		In the polynomial algebra $\F_p[\brxi_1,\cdots,\brxi_r,\ \breta_1,\cdots,\breta_r]$, regarded as an $\F_p$-subalgebra of $H^*(BV^{2r};\F_p)$, the polynomials
		\begin{equation}\label{eq:2r-polys}
			\{\sum_{i=1}^r(\brxi_i^{p^{k+1}}\breta_i-\brxi_i\breta_i^{p^{k+1}})\mid  0\leq k\leq 2r-1\}
		\end{equation}
		are algebraically independent.
	\end{lemma}
	\begin{proof}
		A straightforward computation shows that the Jacobian determinant of the collection of polynomials \eqref{eq:2r-polys} in the variables $\brxi_i,\breta_i$ is
		\begin{equation*}
			J=(-1)^r\opn{det}
			\begin{pmatrix}
				\breta_1^p & \cdots & \breta_r^p & \brxi_1^p & \cdots &\brxi_r^p\\
				\breta_1^{p^2}  & \cdots & \breta_r^{p^2} & \brxi_1^{p^2} & \cdots &\brxi_r^{p^2}\\
				\vdots & & \vdots & \vdots & & \vdots\\
				\breta_1^{p^{2r}}  & \cdots & \breta_r^{p^{2r}} & \brxi_1^{p^{2r}} & \cdots &\brxi_r^{p^{2r}}
			\end{pmatrix},
		\end{equation*}
		which coincides with one of the canonical generators of the Dickson invariant algebra \cite{dickson1911fundamental} of $\F_p[\brxi_1,\cdots,\brxi_r,\ \breta_1,\cdots,\breta_r]$. We have $J\neq0$, since the term
		$\prod_{i=1}^r\brxi_i^{p^i}\cdot\prod_{j=1}^r\breta_i^{p^{r+i}}$ occurs once and once only in its expansion, an observation made at the beginning of Section 3, Chapter \RomanNumeralCaps{3} of \cite{adem2013cohomology}.
		It then follows from the partial Jacobian criterion Proposition \ref{pro:Jacobian} that the polynomials $\{\sum_{i=1}^r(\brxi_i^{p^{k+1}}\breta_i-\brxi_i\breta_i^{p^{k+1}})\mid 0\leq k\leq 2r-1\}$ are algebraically independent.
	\end{proof}
	
	\begin{lemma}\label{lem:chiM-diag-2}
		Suppose $m\mid n$. Let $\Delta: PGL_m\to PGL_n$ be the diagonal map. Then in the category $\Mot$ we have the commutative diagram
		\begin{equation*}
			\begin{tikzcd}
				BPGL_m\arrow[rr,"B\Delta"]\arrow[rd,"\chi_M"]&  & BPGL_n\arrow[ld,"\chi_M"]\\
				& K(\Z(2),3). &
			\end{tikzcd}
		\end{equation*}
		In the category $\Top$ we have a similar commutative diagram
		\begin{equation*}
			\begin{tikzcd}
				BPU_m\arrow[rr,"B\Delta"]\arrow[rd,"\chi"]&  & BPU_n\arrow[ld,"\chi"]\\
				& K(\Z,3). &
			\end{tikzcd}
		\end{equation*}
	\end{lemma}
	\begin{proof}
		Consider the commutative diagram of algebraic groups
		\begin{equation*}
			\begin{tikzcd}
				\mu_m\arrow[r]\arrow[d]& SL_m\arrow[r]\arrow[d,"\Delta"]& PGL_m\arrow[d,"\Delta"]\\
				\mu_n\arrow[r]&     SL_n\arrow[r]&  PGL_n,
			\end{tikzcd}
		\end{equation*}
		which induces commutative diagrams
		\begin{equation*}
			\begin{tikzcd}
				\He^1(-;PGL_m)\arrow[d]\arrow[r,"\Delta_*"] & \He^1(-;PGL_n)\arrow[d]\\
				\He^2(-;\mu_m)\arrow[r] & \He^2(-;\mu_n).
			\end{tikzcd}
		\end{equation*}
		and
		\begin{equation*}
			\begin{tikzcd}
				\left[-,BPGL_m\right]\arrow[d]\arrow[rr,"B\Delta_*"]& & \left[-,BPGL_n\right]\arrow[d]\\
				H_M^{2,2}(-;\Z/m)\arrow[rd,"\delta"]\arrow[rr] &  &H_M^{2,2}(-;\Z/n)\arrow[ld,"\delta"]\\
				& H_M^{3,2}(-;\Z).
			\end{tikzcd}
		\end{equation*}
		The first part of the lemma then follows. The proof for the second part is similar, using the diagrams above with {\'e}tale cohomology and motivic cohomology replaced by singular cohomology.
	\end{proof}
	\begin{proof}[Proof  Theorem \ref{thm:main}]
		The classes $\rho_{p,k}$ and $y_{p,k}$ are constructed in Section \ref{sec:rho} as classes in the images of
		\[\chi_M^*:H_M^*(K(\Z(2),3))\to H_M^*(BPGL_n)\]
		and
		\[\chi^*:H^*(K(\Z,3))\to H^*(BPU_n),\]
		respectively. The periodicity condition \eqref{eq:periodic}, i.e., $B\Delta^*(\rho_{p,k})=\rho_{p,k}$ and $B\Delta^*(y_{p,k})=y_{p,k}$ for the diagonal homomorphism $\Delta: PGL_m\to PGL_n$, follows from Lemma \ref{lem:chiM-diag-2}.
		
		It remains to show that that the homomorphisms \eqref{eq:rho-p-k} and \eqref{eq:y-p-k} are injective. We break the proof into several steps.
		
		\textbf{Step 1.} We prove the injectivity of \eqref{eq:y-p-k} for $n=p^r$. Consider the composite homomorphism
		\begin{equation*}
			\Z[Y_0,\cdots,Y_{2r-1}]/(pY_k)\xrightarrow{\eqref{eq:y-p-k}} H^*(BPU_{p^r}) \xrightarrow{B\theta^*} H^{*}(BV^{2r})\to H^{*}(BV^{2r};\F_p)
		\end{equation*}
		where the last arrow is the mod $p$ reduction. It follows from Corollary \ref{cor:Btheta} and Lemma \ref{lem:alg-ind} that the above homomorphism is injective in degrees above $0$, and we conclude.
		
		\textbf{Step 2.} We prove the injectivity of \eqref{eq:y-p-k} for $n=p^rm$, with $p\nmid m$. By Lemma \ref{lem:chiM-diag-2}, the homomorphism $\chi:H^*(K(\Z,3))\to H^*(BPU_{p^r})$ factors as
		\[H^*(K(\Z,3))\xrightarrow{\chi^*}H^*(BPU_n)\xrightarrow{B\Delta^*} H^*(BPU_{p^r}).\]
		Hence, the homomorphism \eqref{eq:y-p-k}
		\begin{equation*}
			\Z[Y_0,\cdots,Y_{2r-1}]/(pY_k)\to H^*(BPU_{p^r})
		\end{equation*}
		factors as
		\begin{equation*}
			\Z[Y_0,\cdots,Y_{2r-1}]/(pY_k)\xrightarrow{\eqref{eq:y-p-k}} H^*(BPU_n) \xrightarrow{B\Delta^*}  H^*(BPU_{p^r}),
		\end{equation*}
		and we conclude from Step 1.
		
		\textbf{Step 3.} We prove the injectivity of \eqref{eq:rho-p-k}. This follows from the fact that the homomorphism \eqref{eq:y-p-k} factors as
		\begin{equation*}
			\Z[Y_0,\cdots,Y_{2r-1}]/(pY_k)\xrightarrow{\eqref{eq:rho-p-k}}\CH^*(BPGL_n)\xrightarrow{\cl}H^*(BPU_n)
		\end{equation*}
		and we conclude from Step 2.
	\end{proof}
	
	\section{A restriction-transfer calculation}\label{sec:polynomial}
	In this section we prove Theorem \ref{thm:subringR}, which asserts that the subrings 
	\begin{equation*}
		\begin{cases}
			\fR_M(n) = \Z[\rho_{p,k} \mid k\geq 0]/(p\rho_{p,k}) \subset \CH^*(BPGL_n),\\
			\fR(n) = \Z[y_{p,k} \mid k\geq 0]/(py_{p,k}) \subset H^*(BPGL_n).
		\end{cases}
	\end{equation*}
	are determined by the $p$-adic valuation of $n$, and Theorem \ref{thm:example}, which asserts the existence of a nontrivial polynomial relation in $\rho_{p,k}\in\CH^*(BPGL_n)$ (resp. $y_{p,k}\in H^*(BPGL_n)$) for $n$ of $p$-adic valuation $1$ and $k=0,1,2$. Theorem \ref{thm:example} tells us that the role of the $p$-adic valuation of $n$ is essential in Thoerem \ref{thm:main}.
	
	 For $0\leq m \leq n$, we define a subgroup of $SL_n$ as follows:
	\begin{equation*}
		SL_{m,n-m}=\{
		\begin{pmatrix}
			A_1 & 0\\
			0 & A_2
		\end{pmatrix}\
		\in SL_n\mid A_1\in SL_m,\ A_2\in SL_n\}.
	\end{equation*}
	Passing to quotients by centers, we obtain a subgroup $PGL_{m,n-n}$ of $PGL_n$.

	For the rest of this section, we denote by $r$ the $p$-adic valuation of $n$. Then there is a diagonal homomorphism
	\begin{equation}\label{eq:diag-p,n-p}
		\begin{split}
			PGL_{p^r}\to PGL_{p^r,n-p^r},\
			[A_1]\mapsto
			\begin{bmatrix}
				A_1 & &\\
				& \ddots & \\
				& & A_1
			\end{bmatrix},
		\end{split}
	\end{equation}
	together with a left inverse, the projection map
	\begin{equation}\label{eq:proj-p,n-p}
		\begin{split}
			PGL_{p^r,n-p^r}\to PGL_{p^r},\
			\begin{bmatrix}
				A_1 & 0\\
				0 & A_2
			\end{bmatrix}
			\mapsto [A_1].
		\end{split}
	\end{equation}
	Recall the motivic class $\zeta_1\in H_M^{3,2}(BPGL_n)$, which is represented by
	\begin{equation*}
		\chi_M: BPGL_n\to K(\Z(2),3).
	\end{equation*}
	Consider the short exact sequence of algebraic groups
	\begin{equation}\label{eq:PGL-p-n-pSES}
		1\to\mu_{p^r}\to SL_{p^r,n-p^r}\to PGL_{p^r,n-p^r}\to 1.
	\end{equation}
	The procedures \eqref{eq:con-zeta-1}, \eqref{eq:con-zeta-2}, and \eqref{eq:def-zeta} that produce $\zeta_1$ via {\'e}tale cohomology and the Beilinson-Lichtenbaum conjecture may be applied to $PGL_{p^r,n-p^r}$ and yield the following morphism in $\Mot$:
	\begin{equation*}
		\chi'_M: BPGL_{p^r,n-p^r}\to B^2\mu_{p^r}\cong K(\Z/p(2),2)\xrightarrow{\delta} K(\Z(2),3).
	\end{equation*}
	We denote the corresponding class by $\zeta'_1\in H_M^{3,2}(BPGL_{p^r,n-p^r})$.
	\begin{lemma}\label{lem:chiM-diag}
		We have the commutative diagram
		\begin{equation*}
			\begin{tikzcd}
				BPGL_n\arrow[rd,"\chi_M"] & BPGL_{p^r,n-p^r}\arrow[r]\arrow[d,"\chi'_M"]\arrow[l] & BPGL_{p^r}\arrow[ld,"\chi_M"]\\
				& K(\Z(2),3) &
			\end{tikzcd}
		\end{equation*}
		where the horizontal arrows are the ones induced by the obvious homomorphisms of algebraic groups.
	\end{lemma}
	\begin{proof}
		Consider the commutative diagram of algebraic groups
		\begin{equation*}
			\begin{tikzcd}
				\mu_{p^r}\arrow[r]&                                         SL_{p^r}\arrow[r]&                         PGL_{p^r}\\
				\mu_{p^r}\arrow[r]\arrow[u,"\opn{id}"]\arrow[d]& SL_{p^r,n-p^r}\arrow[r]\arrow[u]\arrow[d]& PGL_{p^r,n-p^r}\arrow[u]\arrow[d]\\
				\mu_n\arrow[r]&                                         SL_n\arrow[r]&                         PGL_n,
			\end{tikzcd}
		\end{equation*}
		which induces a commutative diagram in $\Mot$:
		\begin{equation*}
			\begin{tikzcd}
				\left[-,BPGL_n\right]\arrow[d] & \left[-,BPGL_{p^r,n-p^r}\right]\arrow[d]\arrow[l]\arrow[r] &
				\left[-,BPGL_{p^r}\right]\arrow[d] \\
				H_M^{2,2}(-;\Z/n) & H_M^{2,2}(-;\Z/p^r)\arrow[l]\arrow[r,"="] & H_M^{2,2}(-;\Z/p^r)
			\end{tikzcd}
		\end{equation*}
		where  $[-,-]$ is short for $\opn{Hom}_{\Mot}(-,-)$. The desired commutative diagram is obtained once we apply the Bockstein homomorphism to the second row in the diagram above.
	\end{proof}
	
	For an algebraic group or a compact Lie group $G$, let $T(G)$ denote a maximal torus of $G$. Then the normalizers of $T(PGL_{p^r})$, $T(PGL_n)$, $T(PGL_{p^r,n-p^r})$ are respectively the inner semi-direct products
	\begin{equation*}
		\begin{cases}
			\Gamma_{p^r}:=S_{p^r}\ltimes T(PGL_{p^r}),\\
			\Gamma_n:=S_n\ltimes T(PGL_n),\\
			\Gamma_{p^r,n-p^r}:=S_{p^r,n-p^r}\ltimes T(PGL_n), \ \textrm{where }S_{p^r,n-p^r}=S_{p^r}\times S_{p^r,n-p^r}.
		\end{cases}
	\end{equation*}
	Therefore, we have a diagram
	\begin{equation}\label{eq:2-by-3-diag}
		\begin{tikzcd}
			\Gamma_{p^r}\arrow[d]\arrow[r]& \Gamma_{p^r,n-p^r}\arrow[d]\arrow[r]\arrow[l,bend right]& \Gamma_n\arrow[d]\\
			PGL_{p^r}\arrow[r]&             PGL_{p^r,n-p^r}\arrow[r]\arrow[l,bend left]&              PGL_n
		\end{tikzcd}
	\end{equation}
	in which the arrows on the top row are restrictions of the ones on the bottom row. In particular, the straight arrows are inclusions and the bent ones are the projections defined by \eqref{eq:proj-p,n-p}. One easily checks that the diagram \eqref{eq:2-by-3-diag}, without the bent arrows, is commutative.
	
	As there are too many homomorphisms of algebraic/Lie groups in sight, we introduce the following systematic notations. For a homomorphism $H\to G$ which is clear from the context, such as one in the diagram \eqref{eq:2-by-3-diag}, we write
	\begin{equation*}
		\begin{cases}
			\re^G_H: \CH^*(BG)\to\CH^*(BH),\\
			\re^G_H: H^*(BG)\to H^*(BH)
		\end{cases}
	\end{equation*}
	for the restriction homomorphisms.
	
	Next we consider the transfers
	\begin{equation*}
		\begin{cases}
			\tr^H_G: \CH^*(BH)\to\CH^*(BG),\\
			\tr^H_G: H^*(BH)\to H^*(BG)
		\end{cases}
	\end{equation*}
	for $H\hookrightarrow G$ an inclusion of algebraic/Lie groups of finite index. Notice that the transfers are only homomorphisms of graded abelian groups, not ring homomorphisms in general. The transfers and the restriction homomorphisms interact in an intricate way, described by the Mackey's formula (Proposition 4.4, \cite{vistoli2007cohomology}). We are only concerned with a simple special case as follows:
	\begin{lemma}\label{lem:Mackey}
		Let $H\hookrightarrow G$ an inclusion of algebraic/Lie groups of finite index $[G:H]$. Then we have
		\begin{equation*}
			\begin{cases}
				\tr^H_G\cdot\re^G_H=[G:H]\opn{id}:\CH^*(BG)\to\CH^*(BG),\\
				\tr^H_G\cdot\re^G_H=[G:H]\opn{id}: H^*(BG)\to H^*(BG).
			\end{cases}
		\end{equation*}
	\end{lemma}
	
	Another key result is the following
	\begin{theorem}[Gottlieb; Totaro, Theorem 2.1, \cite{vezzosi2000chow}]\label{thm:Gottlieb}
		Let $G$ be an algebraic group over $\C$, $T$ a maximal torus of $G$ and $N(T)$ its normalizer in $G$. The restriction maps
		\begin{equation*}
			\begin{cases}
				\re^G_{N(T)}: \CH^*(BG)\to\CH^*(BN(T)),\\
				\re^G_{N(T)}: H^*(BG)\to H^*(BN(T))
			\end{cases}
		\end{equation*}
		are injective.
	\end{theorem}
	
	We shall now prove Theorem \ref{thm:subringR}.
	
	\begin{theorem*}[Theorem \ref{thm:subringR}]
		Let $p$ be an odd prime and $n>1$ an integer with $p$-adic valuation $r>0$. Then the homomorphisms $B\Delta^*$ restrict to isomorphisms
		\begin{equation*}
			\begin{cases}
				B\Delta^*: \fR_M(n)\xrightarrow{\cong} \fR_M(p^r) ,\ \rho_{p,k}\mapsto \rho_{p,k},\\
				B\Delta^*: \fR(n)\xrightarrow{\cong} \fR(p^r) ,\ y_{p,k}\mapsto y_{p,k}.
			\end{cases}
		\end{equation*}
	\end{theorem*}
	
	\begin{proof}
		We only consider the case for Chow rings. The proof for the case of cohomology is verbatim. 
		
		Let
		\begin{equation*}
			\begin{split}
				&\rho\in\fR_M(n),\ \hat{\rho} = \re^{PGL_n}_{PGL_{p^r}}(\rho)\in \fR_M(p^r),\\
				&u=\re^{PGL_n}_{\Gamma_n}(\rho)\in\CH^*(B\Gamma_n),\ \hat{u} = \re^{\Gamma_n}_{\Gamma_{p^r}}(u)\in \CH^*(B\Gamma_{p^r}).
			\end{split}
		\end{equation*} 
		By Lemma \ref{lem:chiM-diag}, we have
		\begin{equation}\label{eq:rho-hatrho}
			\begin{cases}
				\re^{PGL_{p^r}}_{PGL_{p^r,n-p^r}}(\hat{\rho})=\re^{PGL_n}_{PGL_{p^r,n-p^r}}(\rho),\\
				\re^{\Gamma_{p^r}}_{\Gamma_{p^r,n-p^r}}(\hat{u})=\re^{\Gamma_n}_{\Gamma_{p^r,n-p^r}}(u).
			\end{cases}
		\end{equation}
		Now we have
		\begin{equation}\label{eq:restr}
			\begin{split}
				& {n\choose p^r}u = [\Gamma_n:\Gamma_{p^r,n-p^r}](u) \\
				=& \tr^{\Gamma_{p^r,n-p^r}}_{\Gamma_n}\cdot\re^{\Gamma_n}_{\Gamma_{p^r,n-p^r}}(u) \ \ \ (\textrm{Lemma }                \ref{lem:Mackey})\\
				=& \tr^{\Gamma_{p^r,n-p^r}}_{\Gamma_n}\cdot\re^{\Gamma_{p^r}}_{\Gamma_{p^r,n-p^r}}(\hat{u})       \ \ \ \eqref{eq:rho-hatrho}.\\
			\end{split}
		\end{equation}
		Since $r$ is the $p$-adic valuation, we have ${n\choose p^r}\nmid 0 \pmod{p}$,
		and by \eqref{eq:restr}, we have $u=0$ if $\hat{u} = 0$. The injectivity of $B\Delta^* = \re^{PGL_n}_{PGL_{p^r}}$ follows from Theorem \ref{thm:Gottlieb}. 
	\end{proof}
	
	The following lemma is essentially due to Vistoli \cite{vistoli2007cohomology}.
	\begin{lemma}\label{lem:BPGLp-inj}
		For $p$ and odd prime, consider the subgroup of $\CH^*(BPGL_p)$ of torsion classes, which we denote by $\CH^*(BPGL_p)_{tor}$. The homomorphism
		\[B\theta^*:\CH^*(BPGL_p)\to\CH^*(BV^2)\]
		then restricts to $\CH^*(BPGL_p)_{tor}$. The restriction
		\[B\theta^*:\CH^*(BPGL_p)_{tor}\to\CH^*(BV^2).\]
		is injective. 
	\end{lemma}
	\begin{proof}
		Consider the inclusion $V^2\xrightarrow{B\theta} PGL_p$. We have the homomorphisms induced by the inclusions
		\begin{equation}\label{eq:Visotli-inj}
			\re^{PGL_p}_{T(PGL_p)}\times\re^{PGL_p}_{V^2}: \CH^*(BPGL_p)\to \CH^*(BT(PGL_p))\times\CH^*(BV^2).
		\end{equation}
		Since $\CH^*(BT(PGL_p))$ is torsion-free, \eqref{eq:Visotli-inj} restricted to $\CH^*(BPGL_p)_{tor}$ has the following form:
		\begin{equation}\label{eq:Visotli-inj-res}
			\CH^*(BPGL_p)_{tor}\to \{0\}\times\CH^*(BV^2)
		\end{equation}
		It follows from Proposition 9.4 of \cite{vistoli2007cohomology} that \eqref{eq:Visotli-inj} is injective for $n=p$. Therefore, so is \eqref{eq:Visotli-inj-res}.
	\end{proof}
	We shall now prove Theorem \ref{thm:example}.
	\begin{theorem*}[Theorem \ref{thm:example}]
		For $p$ and odd prime, and $n>0$ an integer satisfying $p\mid n$ and $p^2\nmid n$, the classes $\rho_{p,k}\in\CH^*(BPGL_n)$ for $k=0,1,2$, satisfy a nontrivial polynomial relation
		\begin{equation}\label{eq:polynomial-relation-rho'}
			\rho_{p,0}^{p^2+1}+\rho_{p,1}^{p+1}+\rho_{p,0}^p\rho_{p,2}=0,
		\end{equation}
		and similarly for $y_{p,k}\in H^*(BPU_n)$, $k=0,1,2$, we have
		\begin{equation}\label{eq:polynomial-relation-y'}
			y_{p,0}^{p^2+1}+y_{p,1}^{p+1}+y_{p,0}^py_{p,2}=0.
		\end{equation}
	\end{theorem*}
	\begin{proof}
		We consider only the case for Chow rings. The case for singular cohomology follows from the existence of the cycle class map.
		
		For $n=p$, a routine computation yields
		\begin{equation*}
			B\theta^*(\rho_{p,0}^{p^2+1}+\rho_{p,1}^{p+1}+\rho_{p,0}^p\rho_{p,2})=0,
		\end{equation*}
		and the desired result follows from Lemma \ref{lem:BPGLp-inj}. The general case follows from Theorem \ref{thm:subringR}.
	\end{proof}

	\begin{remark}
		Recall from Theorem \ref{pro:non-toral-inv} that we have
		\[\opn{Im}B\theta^*\subset\CH^*(BV^2)^{\Sp(1)}.\]
		For $n=p$, it is shown by Vistoli (Proposition 5.4, \cite{vistoli2007cohomology}) that the latter is generated, as a ring, by $B\theta^*(\rho_{p,0})$ and a class $q$ satisfying
		\[B\theta^*(\rho_{p,0})q=B\theta^*(\rho_{p,1}).\]
	\end{remark}

	To conclude this section, we observe that Theorem \ref{thm:example} provides an obstruction to the reduction of principal $PGL_n$-bundles. Recall that by the end of Section \ref{sec:mot} we present an isomorphism of functors
	\[\He^1(-;G)\cong \opn{Hom}_{\mathbf{HMot^k_{Nis}}}(-,BG).\]
	With the right-hand side of the above passing to $\Mot$, we obtain a natural transformation
	\[\Theta: \He^1(-;G)\to \opn{Hom}_{\Mot}(-,BG).\]
	
	\begin{proposition}\label{pro:reduction}
		Let $X$ be a smooth scheme over $\C$, $P$ an {\'e}tale principal $PGL_n$-bundle over $X$, with $p^2\mid n$, and $f = \Theta(P):X\to BPGL_n$ the associated morphism in $\Mot$. Let 
		\[\tilde{\rho}_{p,k}=f^*(\rho_{p,k})\in\CH^*(X).\]
		Let $m>0$ be an integer satisfying $p\mid m$, $p^2\nmid m$ and $m\mid n$. If the {\'e}tale principal $PGL_n$-bundle $P$ may be reduced to a  principal $PGL_m$-bundle via the diagonal map $\Delta^*:PGL_m\to PGL_n$, then we have
		\[\tilde{\rho}_{p,0}^{p^2+1}+\tilde{\rho}_{p,1}^{p+1}+\tilde{\rho}_{p,0}^p\tilde{\rho}_{p,2}=0.\]
		A parallel assertion holds for $X$ a CW complex, $P$ a (topological) principal $PU_n$-bundle, and the cohomology classes $\tilde{y}_{p,k}=f^*(y_{p,k})$.
	\end{proposition}
	\begin{proof}
		The proofs for the cases of schemes and CW complexes are parallel, and we only present the proof for schemes.
		
		By the functorial property of $\Theta$, the {\'e}tale principal $PGL_n$-bundle $P$ may be reduced via $\Delta: PGL_m\to PGL_n$ only if 
		$f:X\to BPGL_n$ factors through the diagonal map $\Delta^*:PGL_m\to PGL_n$ as
		\[f:X\xrightarrow{g}BPGL_m\xrightarrow{B\Delta}BPGL_n.\]
		Therefore, we have
		\[\tilde{\rho}_{p,0}^{p^2+1}+\tilde{\rho}_{p,1}^{p+1}+\tilde{\rho}_{p,0}^p\tilde{\rho}_{p,2}=g^*\cdot B\Delta^*(\rho_{p,0}^{p^2+1}+\rho_{p,1}^{p+1}+\rho_{p,0}^p\rho_{p,2}).\]
		By Theorem \ref{thm:example}, we have
		\[B\Delta^*(\rho_{p,0}^{p^2+1}+\rho_{p,1}^{p+1}+\rho_{p,0}^p\rho_{p,2})=0,\]
		and we conclude.
	\end{proof}
	
	\section{The period-index problem}\label{sec:tpip}
	The period-index problem originally concerns the Brauer group of a field $k$ and the degrees of central simple algebras over $k$, which is then generalized to the Brauer group of a scheme and the degrees of Azumaya algebras over it. For more backgrounds on the period-index problem, see \cite{gille2017central} and \cite{grothendieck1968groupe}. Antieau and Williams \cite{antieau2014period}, \cite{antieau2014topological} are the first to consider the topological version of the period-index problem.
	
	The cohomology of $BPU_n$ plays a key role in the study of the topological period-index problem, as demonstrated in \cite{antieau2014topological} and  \cite{gu2019topological}. We refer the reader to \cite{antieau2014period} and \cite{antieau2014topological} for the background of the topological period-index problem. In a nutshell, it concerns a finite CW-complex $Y$ equipped with a cohomology class $\alpha\in H^3(Y)$ and the greatest common divisor of all positive integers $n$ such that there is a homotopy commutative diagram
	\begin{equation}\label{eq:tpip-diag}
		\begin{tikzcd}
			& BPU_n\arrow[d,"\chi"]\\
			Y\arrow[r,"\alpha"]\arrow[ur,dashed,"P"]& K(\Z,3).
		\end{tikzcd}
	\end{equation}
	In this case we say that the principal $PU_n$-bundle $P$ realizes the class $\alpha$. Notice that such a class $\alpha$ is an $n$-torsion class, and for this reason we define the \emph{topological Brauer group} of $Y$ to be the subgroup of torsion classes of $H^3(Y)$, and an element in this group a (topological) Brauer class of $Y$. The torsion order of $\alpha\in H^3(Y)$ is called the \textit{period} of $\alpha$ and denoted by $\opn{per}(\alpha)$. The greatest common divisor of all $n$ such that a homotopy commutative diagram of the form \eqref{eq:tpip-diag} exists is called the index of $\alpha$ and denoted by $\opn{ind}(\alpha)$.
	
	Similarly, we may consider the period-index problem for motivic spaces and {\'e}tale $PGL_n$-torsors.  We may call the torsion subgroup of $H_M^{3,2}(X)$ the \textit{motivic Brauer group} of $X$ and call an element of the motivic Brauer group of $X$ a \textit{motivic Brauer class} of $X$. However, since the natural map
	\[\He^1(X; PGL_n)\to \opn{Hom}_{\Mot}(X, BPGL_n)\]
	is not in general a bijection, the lifting problem in the homotopy category $\Mot$:
	\begin{equation}\label{eq:mot-tpip-diag}
		\begin{tikzcd}
			& BPGL_n\arrow[d,"\chi_M"]\\
			X\arrow[r,"\alpha'"]\arrow[ur,dashed,"P'"]& K(\Z(2),3)
		\end{tikzcd}
	\end{equation}
	is not equivalent to the problem of finding a $PGL_n$-torsor over $X$ representing $\alpha'$. Yet in this section we are able prove an interesting result by working only in the $\mathbf{A}^1$-homotopy category $\Mot$.
	
	The torsion order of $\alpha'$ is called the \textit{period} of $\alpha'$ and denoted by $\opn{per}(\alpha')$, and the greatest common divisor of all $n$ such that there is a homotopy commutative diagram of the form \eqref{eq:mot-tpip-diag} is called the index of $\alpha'$ and denoted by $\opn{ind}(\alpha')$.
	
	So far, the main examples for $\opn{per}(\alpha)\neq\opn{ind}(\alpha)$ are $2d$-skeletons of the Eilenberg-Mac Lane spaces $K(\Z/m,2)$ with a cell decomposition. See \cite{antieau2014topological}, \cite{gu2019topological} and \cite{gu2020topological}. In what follows we suggest an alternative source of examples.
	
	Consider the non-toral $p$-elementary subgroup $V^{2r}$ of $PU_{p^r}$ and the map $\theta:V^{2r}\to PU_{p^r}$ defined in \eqref{eq:theta}. Recall the generator $x_1$ of $H^3(BPU_{p^2})$, and similarly we have the motivic Brauer class of $BPGL_{p^2}$
	\[\zeta_1\in H_M^{3,2}(BPGL_{p^2}).\]
	Finally, we define
	\[\alpha:=B\theta^*(x_1)\in H^3(BV^4),\ \ \alpha':=B\theta^*(\zeta_1)\in H_M^{3,2}(BV^4).\]
	\begin{proposition}\label{pro:BV2p}
		For the motivic Brauer class $\alpha'$ of $BV^4$, we have
		\[\opn{per}(\alpha')=p,\ \ \opn{ind}(\alpha')=p^2.\]
		For the topological Brauer class $\alpha$ of $BV^4$, we have
		\[\opn{per}(\alpha)=p,\ \ \opn{ind}(\alpha)=p^2.\]
	\end{proposition}
	\begin{proof}
		The equation $\opn{per}(\alpha)=p$ follows from Proposition \ref{pro:comparison}, and the equation $\opn{per}(\alpha')=p$ follows from $\opn{per}(\alpha)=p$ and the commutative diagram
		\begin{equation*}
			\begin{tikzcd}
				H_M^{3,2}(BPGL_{p^2})\arrow[r,"B\theta^*"]\arrow[d,"\cl"]& H_M^{3,2}(BV^4)\arrow[d,"\cl"]\\
				H^3(BPU_{p^2})\arrow[r,"B\theta^*"]&                          H^3(BV^4).
			\end{tikzcd}
		\end{equation*}
		For the indices, notice that there is a canonical map $BV^4\to BPGL_n$ as a morphism in the category of simplicial presheaves, and we have the pullback of the universal {\'e}tale principal bundle over $BPGL_n$. Therefore, we have
		\begin{equation}\label{eq:divide-p2}
			\opn{ind}(\alpha'),\ \opn{ind}(\alpha)\mid p^2.
		\end{equation}
		On the other hand, suppose we have an {\'e}tale principal $PGL_n$-torsor $P$ over $BV^4$ representing the class $\alpha$. Then we have a homotopy commutative diagram
		\begin{equation*}
			\begin{tikzcd}
				& BPGL_n\arrow[d,"\chi_M"]\\
				BV^4\arrow[r,"\alpha'"]\arrow[ur,dashed,"\Theta(P)"]& K(\Z(2),3).
			\end{tikzcd}
		\end{equation*}
		for $p\mid n$ and $p^2\nmid n$. This implies
		\begin{equation*}
			\Theta(P)^*(\rho_{p,0}^{p^2+1}+\rho_{p,1}^{p+1}+\rho_{p,0}^p\rho_{p,2})\neq 0,
		\end{equation*}
		which is absurd, by Theorem \ref{thm:example}. The argument for $\alpha\in H^3(BV^4)$ is similar, and we have
		\begin{equation}\label{eq:nmid-p}
			\opn{ind}(\alpha'),\opn{ind}(\alpha)\nmid p.
		\end{equation}
		By \eqref{eq:divide-p2} and \eqref{eq:nmid-p}, we have
		\begin{equation*}
			\opn{ind}(\alpha)=\opn{ind}(\alpha)=p^2.
		\end{equation*}
	\end{proof}
	
	\section{On the Chern subring of $\CH^*(BPGL_n)_{(p)}$}\label{sec:Chern}
	In this section we prove Theorem \ref{thm:Chern-subring}:
	\begin{theorem*}[Theorem \ref{thm:Chern-subring}]
		Let $n>1$ be an integer, and $p$ one of its odd prime divisor. Then the ring $\CH^*(BPGL_n)_{(p)}$ is not generated by Chern classes. More precisely, the class $\rho_{p,0}^i$ is not in the Chern subring for $p-1\nmid i$.
	\end{theorem*}
	\begin{proof}
		By Lemma \ref{lem:chiM-diag-2}, the homomorphism $\chi_M^*:H_M^{*,*}(K(\Z(2),3))\to H_M^{*,*}(BPGL_p)$ factors as
		\[\chi_M^*:H_M^{*,*}(K(\Z(2),3))\xrightarrow{\chi_M^*}H_M^{*,*}(BPGL_n))\xrightarrow{B\Delta^*} H_M^{*,*}(BPGL_p).\]
		Therefore, the class $\rho_{p,0}\in H_M^{*,*}(BPGL_p)$ is in the image of
		\[H_M^{*,*}(BPGL_n)\xrightarrow{B\Delta^*} H_M^{*,*}(BPGL_p),\]
		and the theorem follows from
		\begin{theorem}[Kameko-Yagita, Theorem 1.1 and Theorem 1.3, \cite{kameko2010chern}]\label{thm:Chern-subring-p}
			Let $p$ be an odd prime. Then the ring $\CH^*(BPGL_p)_{(p)}$ is not generated by Chern classes. More precisely, the class $\rho_{p,0}^i$ is not in the Chern subring for $p-1\nmid i$.
		\end{theorem}
	\end{proof}
	

	\appendix
	
	\section{A partial Jacobian criterion over perfect fields of positive characteristics}\label{sec:append}
	For a base field $\F$, we have the Jacobian criterion for the algebraic independence of a collection of polynomials $\{\varphi_i\}$ in the polynomial ring $\F[x_1,\cdots,x_n]$, which is well known to hold in the case that the base field has characteristic $0$, or sufficiently large characteristics relative to the degrees of $\{\varphi_i\}$. We establish a partial Jacobian criterion in the same vein over perfect fields of positive characteristics, which plays a key role in the proof of Lemma \ref{lem:alg-ind}. The criterion may be deduced from, for example, Corollary 16.17 and Corollary A1.7 of Eisenbud \cite{eisenbud2013commutative}. For completeness and simplicity we present an alternative proof.
	\begin{proposition}\label{pro:Jacobian}
		Consider the polynomial algebra $\mathbb{F}[x_1,\cdots,x_n]$, where $\mathbb{F}$ is a perfect field of characteristic $p>0$. Let \[\varphi_1,\cdots,\varphi_m\in\mathbb{F}[x_1,\cdots,x_n],\ m\leq n\]
		be polynomials such that the Jacobian matrix $(\partial\varphi_j/\partial x_i)_{ij}$ is of rank $m$. Then $\varphi_1,\cdots,\varphi_m$ are algebraically independent.
	\end{proposition}
	\begin{proof}
		Suppose $\varphi_1,\cdots,\varphi_m$ are algebraically dependent. Let $f(y_1,\cdots,y_m)$
		be the nontrivial polynomial of the lowest degree such that we have
		\[f(\varphi_1,\cdots,\varphi_m)=0.\]
		Since the Jocobian matrix is of full rank, we have $\partial f/\partial\varphi_i=0$ for all $i$. Therefore, we have
		\[f(\varphi_1,\cdots,\varphi_m)=g(\varphi_1^p,\cdots,\varphi_m^p)\]
		for some polynomial
		\[g(z_1,\cdots,z_m)=\sum_{i_1,\cdots,i_m}a_{i_1,\cdots i_m}z_1^{i_1}\cdots z_m^{i_m}.\]
		Since $\F$ is a perfect field of characteristic $p>0$, we have $b_{i_1,\cdots,i_m}\in\mathbb{F}$ satisfying $b_{i_1,\cdots,i_m}^p=a_{i_1,\cdots,i_m}$.Let
		\[\bar{g}(w_1,\cdots w_m)=\sum_{i_1,\cdots,i_m}b_{i_1,\cdots i_m}w_1^{i_1}\cdots w_m^{i_m}\neq0.\]
		Then we have
		\begin{equation*}
			\begin{split}
				&0=f(\varphi_1,\cdots,\varphi_m)=g(\varphi_1^p,\cdots,\varphi_m^p)\\
				=&\sum_{i_1,\cdots,i_m}(b_{i_1,\cdots i_m}\varphi_1^{i_1}\cdots \varphi_m^{i_m})^p=\bar{g}(\varphi_1,\cdots,\varphi_m)^p.
			\end{split}
		\end{equation*}
		Therefore, $\bar{g}(\varphi_1,\cdots,\varphi_m)=0$ is a nontrivial polynomial relation for $\varphi_1,\cdots,\varphi_m$, and the polynomial $\bar{g}$ has degree lower than that of $f$, a contradiction. Therefore, $\varphi_1,\cdots,\varphi_m$ are algebraically independent.
	\end{proof}
	\bibliographystyle{abbrv}
	\bibliography{RefCH-BPUV4}
\end{document}